\documentclass[reqno]{amsart}

\usepackage{amsmath}
\usepackage{amsfonts}
\usepackage{amsthm}
\usepackage{amssymb}
\usepackage{graphicx}
\usepackage{graphics}
\usepackage{bm}
\usepackage{dsfont}
\usepackage{color}
\usepackage{enumerate}
\usepackage{comment}
\usepackage[font=footnotesize]{caption}
\usepackage[all]{xy}
\usepackage{xypic}
\usepackage{tikz}
\usepackage{paralist}
\allowdisplaybreaks
\usepackage{pgfplots}
\pgfplotsset{%
   every tick label/.append style = {font=\tiny},
   every axis label/.append style = {font=\scriptsize}
}

\numberwithin{equation}{section}

\newtheorem{thm}{Theorem}[section]
\newtheorem*{thm*}{Theorem}
\newtheorem{cor}[thm]{Corollary}
\newtheorem{lem}[thm]{Lemma}

\theoremstyle{definition}
\newtheorem{rem}[thm]{Remark}

\newtheorem{dfn}[thm]{Definition}
\newtheorem{rmk}[thm]{Remark}

\newcommand{\N}{\mathds{N}}

\newcommand{\Z}{\mathds{Z}}

\newcommand{\R}{\mathds{R}}

\newcommand{\T}{\mathds{T}}

\newcommand{\diff}{\mathrm{d}}

\newcommand{\fl}{\mathrm{flat}}

\newcommand{\con}{\mathrm{con}}

\begin{document}

\title[Normal forms for strong magnetic systems]{Normal forms for strong magnetic systems on surfaces:\\trapping regions and rigidity of Zoll systems}

\author[L. Asselle]{Luca Asselle}
\address{Justus Liebig Universit\"at Giessen, Mathematisches Institut \newline \indent Arndtstrasse 2, 35392 Giessen, Germany}
\email{luca.asselle@math.uni-giessen.de}

\author[G. Benedetti]{Gabriele Benedetti}
\address{Universit\"at Heidelberg, Mathematisches Institut, \newline\indent Im Neuenheimer Feld 205, 69120 Heidelberg, Germany}
\email{gbenedetti@mathi.uni-heidelberg.de}

\date{November, 2019}

\begin{abstract}
We prove a normal form for strong magnetic fields on a closed, oriented surface and use it to derive two dynamical results for the associated flow. First, we show the existence of KAM tori and trapping regions provided a natural non-resonance condition holds. Second, we prove that the flow cannot be Zoll unless (i) the Riemannian metric has constant curvature and the magnetic function is constant, or (ii) the magnetic function vanishes and the metric is Zoll. We complement the second result by exhibiting an exotic magnetic field on a flat two-torus yielding a Zoll flow for arbitrarily weak rescalings. 
\end{abstract}

\maketitle

\section{Introduction}
Let $M$ be a closed, oriented surface. A \textit{magnetic system} on $M$ is a pair $(g,b)$, where $g$ is a Riemannian metric on $M$ and $b:M\to\R$ is a function, which we refer to as the magnetic function. A $(g,b)$-\textit{geodesic} is a curve $\gamma:\R\to M$ which is parametrised by arc-length and solves the equation
\begin{equation}\label{e:gbgeo}
\kappa_\gamma(t)=b(\gamma(t)),\qquad\forall\,t\in\R,
\end{equation}
where $\kappa_\gamma$ is the geodesic curvature of $\gamma$.

The term ``magnetic system" refers to the fact that a solution of \eqref{e:gbgeo} describes a trajectory of a particle $\gamma$ with unit charge and speed under the effect of the Lorentz force generated by a stationary magnetic field. To fix ideas, if $M$ is embedded in the euclidean three-dimensional space $\R^3$ and $B:\R^3\to\R^3$ is a magnetic field in the ambient space, then $g$ is the restriction of the euclidean metric on $M$ and $b$ is the inner product of $B$ with the unit normal to $M$ in $\R^3$.

The tangent lifts $(\gamma,\dot\gamma)$ of $(g,b)$-geodesics yield the trajectories of a flow $\Phi_{(g,b)}:SM\times\R\to SM$ on the unit sphere bundle $SM$, whose dynamical properties have been the subject of intensive research since the seminal work of Arnold in the early 60's \cite{Arn61}, see \cite{Ginzburg:1994} for a survey. In the present paper, we will study two important aspects of the dynamics of $\Phi_{(g,b)}$.

The first one is the existence of trapping regions for the flow $\Phi_{(g,b)}$, when the magnetic function is \textit{strong}. Here ``strong" means that $b=\epsilon^{-1}b_1$ for some non-vanishing function $b_1$ and small positive number $\epsilon$. Such a study was initiated by Castilho in \cite{Castilho:2001} and relies on the existence of KAM tori via the Moser twist theorem \cite{Mos62}. We describe this first aspect in detail in Section \ref{ss:kam}. The second one is the existence of 
magnetic systems $(g,b)$ that produce a very simple dynamics, namely that their flow $\Phi_{(g,b)}$ induces a free circle action on $SM$, up to reparametrization. We call such systems \textit{Zoll}. A trivial example of a Zoll system is given by $(g_\mathrm{con},b_{\mathrm{con}})$, where $g_{\mathrm{con}}$ has constant Gaussian curvature $K_{\mathrm{con}}$ and $b_{\mathrm{con}}$ is a constant magnetic function such that $b^2_{\mathrm{con}}+K_{\mathrm{con}}>0$. In this case, $(g_\mathrm{con},b_{\mathrm{con}})$-geodesics are boundaries of geodesic balls and, hence the system is Zoll. Purely Riemannian Zoll systems, namely those with $b=0$, are possible only on $M=S^2$. An infinite-dimensional family of such examples was constructed in \cite{Zol03} among spheres of revolution and in \cite{Gui76}, for the general case.

Key to both problems above is a Hamiltonian normal form for the flow $\Phi_{(g,\epsilon^{-1} b_1)}$. In order to describe it, we first need to see how to associate a Hamiltonian flow on the tangent bundle $TM$ to a general magnetic flow $\Phi_{(g,b)}$. We refer to \cite{Ben14} for the proof of the statements below.

Let $\mu$ be the area form given by the metric $g$ and the orientation on $M$. Moreover, let $\lambda\in\Omega^1(TM)$ be the Hilbert form of $g$:
\[
\lambda_{(q,v)}\cdot\xi=g_q(v,\diff\pi\cdot \xi),\qquad \forall\,(q,v)\in TM,\ \xi\in T_{(q,v)}TM,
\]
where $\pi:TM\to M$ is the canonical projection, $\pi(q,v)=q$. Notice that $\lambda$ is the pull-back of the Liouville one-form on $T^*M$ by 
means of the metric $g$. We define the symplectic form $\omega_{(g,b)}\in\Omega^2(TM)$ on $TM$ via
\[
\omega_{(g,b)}:=\diff\lambda-\pi^*(b\mu)
\]
and the kinetic Hamiltonian function via
\[
H_g:TM\to\R,\qquad H_g(q,v)=\frac12|v|^2_q,\qquad\forall\,(q,v)\in TM,
\]
where $|\cdot|$ is the norm associated with $g$. We denote by $X^{\omega_{(g,b)}}_{H_g}$ the Hamiltonian vector field of $H_{g}$ with respect to $\omega_{(g,b)}$. Let $S_rM:=\{(q,v)\in TM\ |\ |v|_q=r\}$ be the sphere bundle of radius $r>0$. Then $S_rM$ is invariant under the Hamiltonian flow and $X^{\omega_{(g,b)}}_{H_g}|_{S_rM}$ is a nowhere vanishing section of the characteristic distribution $\ker(\omega_{(g,b)}|_{S_rM})\subset T(S_rM)$.

We can now state the promised relationship between the magnetic and the Hamiltonian flow. Denote by $\Gamma:\R\to M$ a curve with constant speed $r=|\dot\Gamma|>0$. Then $(\Gamma,\dot\Gamma)$ is a trajectory of the Hamiltonian flow on ${S_rM}$ if and only if the arc-length reparametrization $\gamma:=\Gamma(\cdot/r)$ is a $(g,\tfrac{1}{r}b)$-geodesic.

Applying the above argument with $r=1$ and $b=\epsilon^{-1} b_1$, we see that $\Phi_{(g,\epsilon^{-1}b_1)}$-geodesics are the integral curves of the characteristic distribution
\[
\ker(\omega_\epsilon|_{SM})\subset T(SM),\qquad \omega_\epsilon:=\epsilon\diff\lambda-\pi^*(b_1\mu)=\epsilon\omega_{(g,\epsilon^{-1} b_1)}.
\]
\subsection{A Hamiltonian normal form}
\mbox{}\medskip

We will find now a Hamiltonian normal form for $\Phi_{(g,\epsilon^{-1} b_1)}$ over an open set $U\subset M$ on which $b_1$ is nowhere vanishing and the circle bundle $\pi:SU\to U$ admits a section. Such a normal form can be seen as the Hamiltonian upgrade of a vector-field normal form for $\Phi_{(g,\epsilon^{-1} b_1)}$ which is due to Arnold \cite[Theorems 2 and 3]{Arn97}, and on which we will comment more at the end of this subsection.

Thus let $W:U\to SU$ be a section and write $\theta:SU\to \T$ for the angle between a unit tangent vector and $W$, measured according to the Riemannian metric and the given orientation on $M$. We refer to $\theta$ as the \textit{angular function} associated with $W$.
\begin{thm}\label{t:nf}
Let $(g,b_1)$ be a magnetic system on a closed, oriented surface $M$. Let $U\subset M$ be an open set on which $b_1$ does not vanish and such that $\pi:SU\to U$ admits
a section $W:U\to SU$. Denote with $\theta:SU\to \T$ the angular function associated with $W$. Then for every open set $U'$ such that $\overline{U'}\subset U$, there exist $\epsilon_0>0$ and an isotopy of embeddings $\Psi_\epsilon:SU'\to SU$, $\epsilon\in[0,\epsilon_0)$, with $\Psi_0$ being the standard inclusion $SU'\hookrightarrow SU$, such that
\[
\Psi_\epsilon^*\omega_\epsilon|_{SM}=\diff(H_\epsilon\diff\theta)-\pi^*(b_1\mu),
\]
where $H_\epsilon:SU'\to\R$ is a path of functions such that
\begin{equation}\label{e:H1}
H_\epsilon=-\frac{\epsilon^2}{2b_1\circ\pi}+o(\epsilon^2).
\end{equation}
If $b_1$ is constant, then we can choose $\Psi_\epsilon$ so that
\begin{equation}\label{e:H2}
H_\epsilon=-\frac{\epsilon^2}{2b_1}-\frac{\epsilon^4}{(2b_1)^3}K\circ\pi+o(\epsilon^4),
\end{equation}
where $K:M\to\R$ is the Gaussian curvature of $g$.
\end{thm}
\begin{rem}
There is an analogous normal form over open sets $U$ of $M$, whose tangent bundle is not trivial. This happens exactly when $U=M$ and $M\neq\mathbb T^2$. In this situation, $\omega_\epsilon|_{SM}=\diff\alpha_\epsilon$ for a path of contact forms $\alpha_\epsilon$ on $SM$. Integrating a suitable vector field $Z_\epsilon$, one finds a path of diffeomorphisms $\Psi_\epsilon:SM\to SM$ such that
\[
\Psi_\epsilon^*\alpha_\epsilon=(1-H_\epsilon)\alpha_0,
\]
where $H_\epsilon$ is of the form \eqref{e:H1} or (when $b_1$ is a non-zero constant) \eqref{e:H2}.
\end{rem}
The normal form up to order $\epsilon^2$ contained in Equation \eqref{e:H1} was already proved by Castilho in \cite[Theorem 3.1]{Castilho:2001} and used to show the existence of KAM tori under certain assumptions on $b_1$, see Section \ref{ss:kam} for more details. A similar normal form was also developed by Raymond and V\~{u} Ng\d{o}c in \cite{Raymond:2015} to study the semi-classical limit of magnetic systems.

Our main contribution in Theorem \ref{t:nf} is, therefore, to push Castilho's normal form to order $\epsilon^4$ when the magnetic function $b_1$ is constant, see Equation \eqref{e:H2}. We shall mention that our method of proof is slightly different from his, as we construct the vector field $Z_\epsilon$ generating the isotopy $\Psi_\epsilon$ at once using Moser's method \cite{Mos65}: $Z_\epsilon$ is the unique vector field contained in the distribution $\mathcal H$ tangent to the level sets of $\theta$ satisfying
\[
\omega_\epsilon(Z_\epsilon,\cdot\,)=-\lambda-\diff c_\epsilon \qquad\text{on }\mathcal H
\]
for some function $c_\epsilon:SU\to\R$ that we will suitably choose.

Let us now comment on the first dynamical implications of Theorem \ref{t:nf}. This result tells us that we can read off $(g,\epsilon^{-1}b_1)$-geodesics as trajectories of a time-dependent Hamiltonian flow on $U'$. To this purpose, let us write by $H_{\epsilon,\theta}:U'\to\R$, $\theta\in \T$, the function $H_{\epsilon,\theta}(q):=H_\epsilon( e^{i\theta}W(q))$, where $e^{i\theta}:SU\to SU$ is the fiberwise rotation by angle $\theta$. We define the $\theta$-dependent vector field $X_{H_{\epsilon,\theta}}$ on $U'$ by
\begin{equation}\label{e:X1}
b_1\mu(X_{H_{\epsilon,\theta}},\cdot\,)=-\diff H_{\epsilon,\theta}
\end{equation}
and the vector field $X_{H_\epsilon}$ on $SU'$ by
\begin{equation}\label{e:X2}
X_{H_\epsilon}=\tilde X_{H_{\epsilon,\theta}}+\partial_\theta,
\end{equation}
where $\tilde X_{H_{\epsilon,\theta}}$ denotes the lift of $X_{H_{\epsilon,\theta}}$ tangent to the level sets of the angular function. If $\Phi_{H_\epsilon}^t$ is the flow of $X_{H_\epsilon}$ on $SU'$, then for all $\theta_0\in \T$ and $\theta_1\in\R$
\[
\varphi^{\theta_0,\theta_1}_\epsilon:=\pi\circ\Phi_{H_\epsilon}^{\theta_1}\circ e^{i\theta_0}W 
\]
is the flow of $X_{H_{\epsilon,\theta}}$ on $U'$ and
\begin{equation}\label{e:theta+}
\theta\big(\Phi^{\theta_1}_{H_\epsilon}( e^{i\theta_0}W)\big)=\theta_0+\theta_1.
\end{equation}
We see that $X_{H_\epsilon}$ spans the kernel of $-\pi^*(b_1\mu)+\diff(H_\epsilon\diff\theta)$. Therefore, $\Psi_\epsilon$ sends trajectories of the flow $\Phi_{H_\epsilon}$ to tangent lifts of $(g,\epsilon^{-1}b_1)$-geodesics, up to reparametrization. More precisely, let $z\in SU'$ and consider $\eta_{z}:[0,\theta_\infty)\to SU'$ to be the maximal solution of the flow $\Phi_{H_\epsilon}$ with $\eta_{z}(0)=z$. Let $\gamma_{\Psi_\epsilon(z)}:\R\to M$ be the $(g,\epsilon^{-1}b_1)$-geodesic with $\dot\gamma_{\Psi_\epsilon(z)}(0)=\Psi_\epsilon(z)$. There exists an increasing function $t:[0,\theta_\infty)\to[0,\ell_\infty)$ such that
\begin{equation}\label{e:gammaeta}
\dot\gamma_{\Psi_\epsilon(z)}(t(\theta))=\Psi_\epsilon(\eta_z(\theta)),\qquad\forall\,\theta\in[0,\theta_\infty).
\end{equation}

From Formulae \eqref{e:H1} and \eqref{e:H2} for $H_\epsilon$, we see that $(g,\epsilon^{-1}b_1)$-geodesics \begin{enumerate}[(i)]
\item follow for a long time non-degenerate level sets of $b_1$ with a drift velocity proportional to $\epsilon^2|\diff (b_1^{-2})|$;
\item follow for a long time non-degenerate level sets of $K$, if $b_1$ is a non-zero constant, with a drift velocity proportional to $\epsilon^4b_1^{-4}|\diff K|$.
\end{enumerate}
Such a dichotomy follows already from Arnold's normal form for vector fields cited above \cite{Arn97}, which shows that the magnetic function $b_1$ or the Gaussian curvature $K$, if $b_1$ is constant, are \textit{adiabatic invariants} for the flow $\Phi_{(g,\epsilon^{-1}b_1)}$. Roughly speaking, a quantity $I$ is an adiabatic invariant if, for some $\alpha>0$ and any threshold $\delta>0$, there exists an $\epsilon_\delta>0$ such that for all $\epsilon<\epsilon_\delta$ the change of $I$ along solutions with parameter $\epsilon$ is smaller than $\delta$ for a time-interval of length $\epsilon^{-\alpha}$ (see \cite[Section 52]{Arnold:1978aq} for more details). We remark, however, that the normal form for vector fields is not enough for the two applications considered in this paper and we crucially need the Hamiltonian normal form contained in Theorem \ref{t:nf}.

We conclude this subsection by observing that the theory of adiabatic invariants for magnetic fields in three-dimensional euclidean space plays an important role in plasma physics as it can be applied to confine charged particles in some region of space; see \cite{Mar20} and the references therein. We refer to \cite{BS94} for a mathematical treatment of magnetic adiabatic invariants in euclidean space and to \cite{HKRV16} for their applications to the spectrum of the magnetic Laplacian.

\subsection{Application I: trapped motions}\label{ss:kam}
\mbox{}\medskip

Throughout this subsection, we denote by $(g,b_1)$ a magnetic system on a closed, oriented surface $M$ such that $b_1:M\to(0,\infty)$ is positive. We also denote by $\mu$ the area form given by $g$ and the orientation on $M$. To write the statements in a more compact form, we define the function
\begin{equation}\label{e:zeta}
\zeta:M\to\R,\qquad \zeta=\begin{cases}
K&\text{if } b_1 \text{ is constant},\\
b_1^{-2}&\text{otherwise}.
\end{cases}
\end{equation}
Moreover, we will denote by $L_{c_0}$ a non-empty connected component of a regular level set $\{\zeta=c_0\}$ for some $c_0\in\R$. In particular, $L_{c_0}$ is an embedded circle belonging to some family of embedded circles $c\mapsto L_{c}$ with $c\in(c_0-\delta_0,c_0+\delta_0)$ for a suitable $\delta_0>0$. There are action-angle coordinates $(I,\varphi)\in(I_-,I_+)\times \T$ around $L_{c_0}$ for the symplectic form $\mu$ such that $\zeta=\tilde\zeta\circ I$ for some function $\tilde\zeta:(I_0-\delta_0',I_0+\delta_0')\to (c_0-\delta_0,c_0+\delta_0)$ with $\tilde\zeta(I_0)=c_0$ and some $\delta_0'>0$.
The circle $L_{c_0}$ is said to be \textit{non-resonant} if
\begin{equation}\label{e:nr}
\frac{\diff^2 \tilde\zeta}{\diff I^2}(I_0)\neq0.
\end{equation}
Combining the Hamiltonian normal form in Theorem \ref{t:nf} with the Moser twist theorem \cite{Mos62}, one can show that $(g,\epsilon^{-1}b_1)$-geodesics are trapped in neighborhoods of non-resonant circles. 
\begin{thm}\label{t:kam}
Let $L_{c_0}$ be non-resonant and let $U$ be a neighborhood of $L_{c_0}$. Then there exists $\epsilon_0>0$ and a neighborhood $U_1\subset U$ of $L_{c_0}$ such that, for every $\epsilon\in(0,\epsilon_0)$, all $(g,\epsilon^{-1}b_1)$-geodesics with starting point in $U_1$ remain in $U$ for all times.
\end{thm}
Following Castilho \cite[Corollary 1.3]{Castilho:2001}, one can use Theorem \ref{t:kam} to prove trapping of $(g,\epsilon^{-1}b_1)$-geodesics around circles $L$ of minima (or maxima) for $\zeta$. The key idea is that $L$ is approximated on either sides by non-resonant circles.
\begin{cor}\label{c:mincircle}
Let $L\subset M$ be an embedded circle. Suppose that $L$ is an isolated critical set of $\zeta$ consisting of local minima or maxima. Then for all neighborhoods $U$ of $L$, there exist $\epsilon_0>0$ and a neighborhood $U_1\subset U$ of $L$ such that, for all $\epsilon\in(0,\epsilon_0)$, all $(g,\epsilon^{-1}b_1)$-geodesics with starting point in $U_1$ remain in $U$ for all times.
\end{cor}

We now want to give some criteria for the existence of non-resonant circles. One classical idea is to look for such circles around a non-degenerate local minimum (maximum) point $q_*\in M$ of $\zeta$ by means of the Birkhoff normal form. Birkhoff normal form tells us that $\zeta$ can be written in a suitable Darboux chart centered at $q_*$ as
\[
\zeta(q)=\zeta(q_*)\pm\tfrac12\sqrt{\det \mathrm{Hess}\, \zeta(q_*)} \, r^2+a(q_*) r^4+o(r^4),
\]
where $r$ is the radial coordinate and $a(q_*)\in \R$ is a real number. Here $\det \mathrm{Hess}\, \zeta\, (q_*)$ is the determinant of the Hessian of $\zeta$ in Darboux coordinates with respect to $\mu$, which coincides with the determinant of the Hessian with respect to $g$.
Since $I=\tfrac12r^2$ is the action variable in Darboux coordinates, we see that
\begin{equation}\label{e:limI}
\lim_{I\to 0}\frac{\diff^2\tilde \zeta}{\diff I^2}\neq 0,
\end{equation}
as soon as $a(q_*)\neq0$. In this case, small circles $L_c$ around $q_*$ will be non-resonant. We mention here an equivalent way to check \eqref{e:limI} using coordinates $(x,y)$ around $q_*$ where $\zeta=\zeta(q_*)+\tfrac12r^2$, which exist thanks to the Morse lemma. Let $\rho$ be the unique function satisfying $\mu=\rho \diff x\wedge\diff y$. We will see in Lemma \ref{l:minpoint} that \eqref{e:limI} holds if and only if the Laplacian $\Delta \rho(q_*)$ of $\rho$ at $q_*$ in the $(x,y)$-coordinates does not vanish.  

On the other hand, when the local minimum (maximum) point $q_*$ is isolated but degenerate, we will see that the existence of non-resonant circles accumulating at $q_*$ automatically follows.
\begin{cor}\label{c:minpoint} 
Let $q_*\in M$ be an isolated local minimum (maximum) point for $\zeta$ which is either degenerate or is non-degenerate and satisfies $a(q_*)\neq0$, equivalently $\Delta\rho(q_*)\neq0$. Then for all neighborhoods $U$ of $q_*$, there exist $\epsilon_0>0$ and a neighborhood $U_1\subset U$ of $q_*$ such that, for all $\epsilon\in(0,\epsilon_0)$, all $(g,\epsilon^{-1}b_1)$-geodesics with initial starting point in $U_1$ remain in $U$ for all times.
\end{cor}
We now present two other situations in which non-resonant circles can be found. The first one uses a saddle point for $\zeta$.
\begin{thm}\label{t:saddle}
Suppose that $\zeta$ has a non-degenerate saddle critical point $q_*\in M$ such that there are no critical values of $\zeta$ in the interval $(\zeta(q_*),\zeta(q_*)+\delta)$ for some $\delta>0$. Then there exists a sequence of embedded circles $L_{c_n}$ such that $L_{c_n}$ is non-resonant and
\[
\lim_{n\to\infty}\mathrm{dist}(q_*,L_{c_n})=0.
\]
\end{thm}
The second one works when the closed, oriented surface $M$ is the two-sphere and the function $\zeta$ has exactly one minimum and one maximum point which, by Corollary \ref{c:minpoint}, can be both assumed to be non-degenerate.
\begin{thm}\label{t:S2}
Suppose that $M=S^2$ and that $\zeta$ has one non-degenerate minimum at $q_{\min}\in S^2$, one non-degenerate maximum at $q_{\max}\in S^2$ and no other critical point. If
\[
\sqrt{\det \mathrm{Hess}\, \zeta (q_{\min})}\neq\sqrt{\det \mathrm{Hess} \, \zeta(q_{\max})},
\]
then there is a non-resonant circle $L_c$. Here the Hessian is taken with respect to the metric $g$.
\end{thm}

\subsection{Application II: Rigidity and flexibility of Zoll systems}\label{ss:Zoll}
\mbox{}\medskip

The only known Zoll pairs $(g,b)$, which are different from the trivial examples $(g_\con,b_\con)$, or the purely Riemannian examples with $b=0$ on $S^2$, were constructed in \cite{Asselle:2019a}. These exotic Zoll pairs are defined on $\T^2=\T\times\T$ with angular coordinates $(x,y)$. Their metric $g$ is flat and their magnetic function $b$ depends on the $x$-variable only.

Unlike the trivial and the purely Riemannian examples, such Zoll pairs do not remain Zoll if we rescale $b$ by an arbitrary constant $r>0$. Therefore, we would like to understand, if there are examples of magnetic systems $(g,b)$ such that $(g,r^{-1}b)$ is Zoll for different values of $r$ or, more generally, for values of $r$ belonging to some given set. This corresponds to asking that the Hamiltonian flow of $H_g$ on the twisted tangent bundle $(TM,\omega_{(g,b)})$ is Zoll at several energy levels $\tfrac12r^2$. To better handle this question, we introduce the following definition.

\begin{dfn}
Let $\mathcal M$ be a family of magnetic systems on a closed, oriented surface $M$ and let $R$ be a subset of $(0,\infty)$. We say that $\mathcal M$ is \textit{Zoll-rigid at} $R$ provided the following holds: If $(g,b)\in\mathcal M$ is a magnetic system such that $(g,r^{-1} b)$ is Zoll
for every $r \in R$, then $g$ has constant curvature and $b$ is constant, 
or $M=S^2$, $g$ is a Zoll metric and $b$ is identically zero. If $\mathcal M$ is not Zoll-rigid at $R$, we say that $\mathcal M$ is \textit{Zoll-flexible at} $R$.
\end{dfn}

In \cite{Asselle:2019b}, the first rigidity phenomena were discovered:
\begin{itemize}
\item[i)] Magnetic systems on $\T^2$ are Zoll-rigid at any set $R$ which accumulates to zero and infinity;
\item[ii)] Magnetic systems on surfaces with genus at least two and with Ma\~n\'e critical value of the universal cover equal to $c$ are Zoll-rigid at any $R\subset (0,\sqrt{2c})$ which accumulates to $\sqrt{2c}$. 
\end{itemize}

In this paper we push the study of rigidity further and prove the following statement. 
\begin{thm}\label{t:Zoll1}
	Magnetic systems on a closed, oriented surface $M$ are Zoll-rigid at any set $R$ which accumulates to zero. Namely, if $(g,b_1)$ is a magnetic system on $M$ such that $(g,\epsilon^{-1}_nb_1)$ is Zoll for some sequence $\epsilon_n\to 0$, then: Either $b_1=0$, $M=S^2$ and $g$ is a Zoll metric, or $b_1$ is a non-zero constant and $g$ is a metric of constant curvature. 
\end{thm}
The proof of Theorem \ref{t:Zoll1} hinges on the normal form proved in Theorem \ref{t:nf}. Roughly speaking, if $b_1$ is not constant, we can suppose, up to changing the sign, that $b_1$ has a positive maximum and apply the normal form to the regions $U=U_1=\{b_1>\delta_1\}$ and $U'=U_2=\{b_1>\delta_2\}$ for some positive $\delta_1<\delta_2$ in the interval $(\min b,\max b)$. Using a theorem of Ginzburg \cite{Ginzburg:1987lq}, we can find $U_3$ such that there exists $q_0\in U_3$ with the following property: the flow line of $\Phi_{H_\epsilon}$ starting at $W(q_0)$ has period $2\pi$. If now $(g,\epsilon^{-1}b_1)$ is Zoll, a topological argument shows that for all $q\in U_3$ the flow line of $\Phi_{H_\epsilon}$ starting at $W(q)$ has period $2\pi$. However, since $b_1$ is not constant, the normal form tells us that there must be $q_1\in U_3$ such that the flow line of $\Phi_{H_\epsilon}$ starting at $W(q_1)$ drifts with speed of order $\epsilon^2$ along a regular level set of $b_1$ inside $U_3$. Hence this flow line cannot close up in time $2\pi$ if $\epsilon$ is small enough. Thus $(g,\epsilon^{-1}b_1)$ cannot be Zoll for $\epsilon$ small enough, when $b_1$ is not constant. If now $b_1$ is a non-zero constant but the Gaussian curvature $K$ of $g$ is non-constant, we can apply a similar argument to the regions $U_1=\{K>\delta_1\}$ and $U_2=\{K>\delta_2\}$ to some $\delta_1<\delta_2$ in the interval $(\min K,\max K)$ to get that $(g,\epsilon^{-1}b_1)$ cannot be Zoll for $\epsilon$ small.
\begin{rmk}
For the restricted case of rotationally invariant systems a slight variant of Theorem \ref{t:Zoll1} was proved by Kudryavtseva and Podlipaev in \cite{KP19}, where rigidity is established provided all orbits are periodic (a condition which is weaker than being Zoll) for every sufficiently small speed (whereas the Zoll condition in Theorem \ref{t:Zoll1} only holds on a sequence of speeds converging to zero). It is natural to expect that such a statement should hold also for general magnetic systems on surfaces using the techniques of the present paper. As observed in \cite{KP19}, results such Theorem \ref{t:Zoll1} can be thought as magnetic analogues of a classical result of Bertrand \cite{Ber} asserting that the only attracting central forces for which every bounded orbit is periodic are the harmonic and the gravitational one \cite[Section 8D]{Arnold:1978aq}.
\end{rmk}
An intriguing problem is to understand exactly which sets guarantee rigidity and which ones flexibility. The threshold between the two behaviors can be subtle as the next result shows. We first define the sets
\[
R_*:=\{\xi\in(0,\infty)\ |\ J_1(\xi)=0\},\quad \qquad R_{\N}:=\bigcup_{k\in\N}\tfrac{1}{k}R_*,
\]
where $J_1$ is the first Bessel function. The set $R_*$ is discrete, bounded away from zero, unbounded from above and asymptotic to the arithmetic progression $\{\frac{3}{4}\pi+2\pi k\ |\ k\in\N\}$. In particular, $R_{\N}$ is a countable dense subset of $(0,\infty)$.
\begin{thm}
Let $\mathcal M$ be the set of rotationally invariant magnetic systems on $\T^2$ with average $1$. Namely, $(g,b)\in \mathcal M$ if and only if $g=\diff x^2+a(x)^2\diff y^2$ and $b=b(x)$ for some functions $a,b:\T\to\R$ such that $\int_{\T}ab\diff x=\int_{\T}a\diff x$. 
The family $\mathcal M$ is
\begin{enumerate}[(a)]
\item Zoll-rigid at every set $R$ which is unbounded from above and not contained in $R_\N$;
\item Zoll-flexible at each of the sets $\tfrac{1}{k}R_*$, $k\in\N$.
\end{enumerate}
\label{t:Zoll2}
\end{thm}
\begin{rmk}
By \cite{Benedetti:2018c}, there are no rotationally invariant Zoll systems on $\T^2$ with zero average. Thus up to rescaling $b$ by a constant, there is no loss of generality in assuming that the average is $1$.
\end{rmk}
Theorem \ref{t:Zoll2} is rather surprising and represents the first instance where the structure of the space of Zoll 
magnetic systems is influenced by the topology of the surface. Indeed, since the Ma\~n\'e critical value of the universal cover for magnetic systems with non-zero average is infinite, we can 
reformulate Theorem \ref{t:Zoll2}, b) by saying that
magnetic systems on $\T^2$ with average $1$ are Zoll-flexible for a particular set accumulating at the Ma\~n\'e critical value of the universal cover. This is in sharp contrast with the rigidity for surfaces with genus at least two described in the item ii) above \cite{Asselle:2019b}. 

%

\subsection*{Structure of the paper.}
\mbox{}\medskip

In Section \ref{s:preliminaries} we present some classical facts about the geometry of surfaces that serve as preliminaries to the Hamiltonian normal form contained in Theorem \ref{t:nf} which will be proved in Section \ref{s:nf}. In Section \ref{s:kam} we establish the existence of trapping regions and prove Theorem \ref{t:kam} and its Corollaries \ref{c:mincircle} and \ref{c:minpoint}. The criteria contained in Theorem \ref{t:saddle} and Theorem \ref{t:S2} for the existence of non-resonating circles will also be discussed there. Section \ref{s:zoll1} deals with the proof of Theorem \ref{t:Zoll1} about the Zoll-rigidity of strong magnetic fields, while Section \ref{s:zoll2} shows Theorem \ref{t:Zoll2} about the rigidity versus flexibility behaviour for rotationally symmetric magnetic fields on the two-torus.  

\subsection*{Acknowledgments.}
\mbox{}\medskip

L.A.~is partially supported by the Deutsche Forschungsgemeinschaft (DFG,
German Research Foundation) under the DFG-grant AS 546/1-1 - 380257369 (Morse theoretical methods in Hamiltonian dynamics).  G.B.~is partially supported by the Deutsche Forschungsgemeinschaft (DFG,
German Research Foundation) under Germany's Excellence Strategy
EXC2181/1 - 390900948 (the Heidelberg STRUCTURES Excellence Cluster), under the Collaborative Research Center SFB/TRR 191 - 281071066 (Symplectic Structures in Geometry, Algebra and Dynamics) and under the Research Training Group RTG 2229 - 281869850 (Asymptotic Invariants and Limits of Groups and Spaces).

\section{Preliminaries from the differential geometry of surfaces}
\label{s:preliminaries}

In this section, we recall some facts from the Riemannian geometry of surfaces that will be useful later on. Let $U\subset M$ be an open set and assume that we have a section $W:U\to SU$ of $\pi:SU\to U$. We denote by $e^{i\theta}:SU\to SU$ the flow of fiberwise rotations and use the shorthand $(\cdot)^\perp:=e^{i\pi/2}(\cdot)$. We denote by $\partial_\theta$ the vector field generating the flow of rotations. Moreover, we denote by $\mathcal H:=\ker\diff\theta\subset T(SU)$ for the distribution tangent to the level sets of $\theta$. We have
\[
T(SU)=\mathcal H\oplus (\R\cdot\partial_\theta).
\]
If $u\in TU$, we denote by $\tilde u$ the unique element in $\mathcal H$ such that $\diff\pi[\tilde u]=u$. 
 
A key object in our computations will be the $1$-form $\alpha\in\Omega^1(U)$ given by
\begin{equation}\label{e:alpha}
\alpha=g(\nabla_{\cdot} W,W^\perp),
\end{equation}
where $\nabla$ is the Levi-Civita connection of $g$. The rotations commute with the connection, so that for every $\theta\in \T$ we have
\begin{equation}\label{e:alphatheta}
\alpha=g(\nabla_{\cdot} W,W^\perp)=g(\nabla_{\cdot} (e^{i\theta}W),(e^{i\theta}W)^\perp).
\end{equation}
 
Let $\hat\lambda$ be the pullback on $SM$ of $\lambda$. Our first task is to compute $\diff\hat\lambda$ on $SU$ in terms of $\alpha$.
\begin{lem}\label{l:dlambda}
There holds
\[
(\diff \hat\lambda)_v=(\pi^*\alpha+\diff\theta)\wedge g(\diff\pi[\,\cdot\,],v^\perp),\qquad\forall\,v\in SU.
\]
\end{lem}
\begin{proof}
Let $v=e^{i\theta}W(q)$ for some $\theta\in \T$ and $q\in U$. The formula for $(\diff\hat\lambda)_v$ has to be checked for the pairs $\tilde u_1,\tilde u_2\in \mathcal H_v$ and $\partial_\theta,\tilde u\in \mathcal H$. Let $u_i:=\diff\pi[\tilde u_i]$ for $i=1,2$. We can suppose without loss of generality that $u_1,u_2$ are defined in a neighborhood of $q$ and satisfy $[u_1,u_2]=0$ there. We compute first
\begin{align*}
\tilde u_1\big(g(u_2,e^{i\theta}W)\big)=g(\nabla_{u_1}u_2,v)+g(u_2,\nabla_{u_1}(e^{i\theta}W))&=g(\nabla_{u_1}u_2,v)+g(u_2,v^\perp)g(e^{i\theta}W^\perp,\nabla_{u_1}(e^{i\theta}W))\\
&=g(\nabla_{u_1}u_2,v)+g(u_2,v^\perp)g(W^\perp,\nabla_{u_1}W)\\
&=g(\nabla_{u_1}u_2,v)+g(u_2,v^\perp)\alpha(u_1),
\end{align*}
where in the second equality we used the fact that $v^\perp = e^{i\theta}W^\perp$ and the identity 
\[
g(\nabla_\cdot( e^{i\theta}W),e^{i\theta}W)=\tfrac12\diff|e^{i\theta}W|^2=\tfrac12\diff(1)=0,
\]
and Identity \eqref{e:alphatheta} in the third equality.

The value of $\diff\hat\lambda(\tilde u_1,\tilde u_2)$ is the antisymmetrization of the expression above since $[u_1,u_2]=0$. Taking into account that $\nabla$ is symmetric we obtain
\[
(\diff\hat\lambda)_v(\tilde u_1,\tilde u_2)=g(u_2,v^\perp)\alpha(u_1)-g(u_2,v^\perp)\alpha(u_1)=\pi^*\alpha\wedge g(\diff\pi[\,\cdot\,],v^\perp)(\tilde u_1,\tilde u_2).
\]
For the pair $\partial_\theta,\tilde u$ we similarly get
\[
(\diff\hat\lambda)_v(\partial_\theta,\tilde u)=\partial_\theta(g(u,e^{i\theta}W))-0=g(u,\partial_\theta(e^{i\theta}W))=g(u,e^{i\theta}W^\perp)=g(u,v^\perp)=\diff\theta\wedge g(\diff\pi[\,\cdot\,],v^\perp)(\partial_\theta,\tilde u).\qedhere
\]

\end{proof}

We now link the Gaussian curvature $K:U\to\R$ of $g$ to the function
\begin{equation}\label{e:f}
f:SU\to \R,\qquad f(v):=\alpha(v)^2+\diff\Big(\alpha\big(e^{i\theta(v)}W^\perp\big)\Big)\cdot v.
\end{equation}
\begin{lem}\label{l:K}
For every $v\in SU$ we have
\begin{equation}\label{e:K1}
K(\pi(v))=f(v)+f(v^\perp).
\end{equation}
In particular, there exists a function $a:SU\to \R$ such that
\begin{equation}\label{e:K2}
\frac12K\circ\pi=f+\partial_\theta a.
\end{equation}
\end{lem}
\begin{proof}
For every fixed $\theta\in \T$, we first check the formula
\[
[e^{i\theta}W,e^{i\theta}W^\perp]=-\alpha(e^{i\theta}W)e^{i\theta}W-\alpha(e^{i\theta}W^\perp)e^{i\theta}W^\perp.
\]
It is enough to prove the result for $\theta=0$. Using the symmetry of $\nabla$, we get
\[
g([W,W^\perp],W)=g(\nabla_WW^\perp,W)+g(\nabla_{W^\perp}W,W)=-g(\nabla_WW,W^\perp)+0=-\alpha(W)
\]
and analogously $g([W,W^\perp],W^\perp)=-\alpha(W^\perp)$.

To compute the Gaussian curvature, we use the classical identity $K\mu=\diff\alpha$, so that 
\begin{equation}\label{e:Kvv}
K(\pi(v))=\diff\alpha(v,v^\perp).
\end{equation}
Let us now set $\theta:=\theta(v)$. Then the vector fields $e^{i\theta}W$ and $e^{i\theta}W^\perp$ extend the vectors $v$ and $v^\perp$ on the whole $U$. Therefore, we can expand the right-hand side of \eqref{e:Kvv} using \eqref{e:alphatheta}:
\begin{align*}
\diff\alpha(v,v^\perp)&=\diff\big(\alpha(e^{i\theta}W^\perp)\big)\cdot v-\diff\big(\alpha(e^{i\theta}W)\big)\cdot v^\perp-\alpha([e^{i\theta}W,e^{i\theta}W^\perp])\\
&=\diff\big(\alpha(e^{i\theta}W^\perp)\big)\cdot v+\diff\big(\alpha(e^{i(\theta+\pi/2)}W^\perp)\big)\cdot v^\perp +\alpha(v)^2+\alpha(v^\perp)^2\\
&=f(v)+f(v^\perp),
\end{align*}
so that we arrive at \eqref{e:K1}. Fixing $q\in U$ and averaging Identity \eqref{e:K1} over $\pi^{-1}(q)$, we get
\[
K(q)=\frac{1}{2\pi}\int_0^{2\pi}\Big(f\big(e^{i\theta}W(q)\big)+f\big(e^{i(\theta+\pi/2)}W(q)\big)\Big)\diff\theta=\frac{1}{\pi}\int_0^{2\pi}f\big(e^{i\theta}W(q)\big)\diff\theta.
\]
It follows that the function
\[
f_1:SU\to\R,\qquad f_1:=\frac12K\circ\pi-f,  
\]
has zero average along the circle $\pi^{-1}(q)$ for all $q\in U$. Because of this fact, the function
\[
a:SU\to \R,\quad a(v):=\int_0^{\theta(v)} f_1\big(e^{i\theta'}W(\pi(v))\big)\diff\theta'
\]
is well-defined and it is immediate to check that $\partial_\theta a=f_1$, as required.
\end{proof}


\section{The construction of the normal form}\label{s:nf}
This section is entirely dedicated to the proof of Theorem \ref{t:nf}. Recalling that $\hat\lambda$ is the restriction to $SM$ of $\lambda$ we write
\[
\hat\omega_\epsilon:=\epsilon\diff\hat\lambda-\pi^*(b_1\mu) 
\]
for the pullback of $\omega_\epsilon$ to $SM$.

We consider $U\subset M$ an open set on which we have a section $W$ of $\pi:SU\to U$ with associated angular function $\theta$. We further assume that $b_1$ is nowhere vanishing on $U$. We take an open set $U'$ such that $\overline{U'}\subset U$. Then there exists a positive number $r>0$ such that any curve $\gamma:[t_0,t_1]\to M$ with $\gamma(t_0)\in U'$ and length less than $r$ is entirely contained in $U$.
 
Let $\Psi_\epsilon:SU'\to SU$ denote an isotopy for $\epsilon$ in some interval $[0,\epsilon_0]$ such that $\Psi_0$ is the standard inclusion $SU'\hookrightarrow SU$. 
Up to shrinking $\epsilon_0$ further, by the definition of $r>0$ we see that $\Psi_\epsilon$ is obtained integrating an $\epsilon$-dependent vector field $Z_\epsilon$ on $SU$.
We aim at finding $Z_\epsilon$ such that
\begin{equation}\label{e:pb}
\Psi_\epsilon^*\hat\omega_\epsilon=-\pi^*(b_1\mu)+\diff(H_\epsilon\diff\theta) \qquad\text{on}\ SU'
\end{equation}
for some $H_\epsilon:SU\to\R$ with $H_0=0$, satisfying either \eqref{e:H1}, or \eqref{e:H2} if $b_1$ is constant on $U$.

Taking the derivative of \eqref{e:pb} in $\epsilon$ and using that $\partial_\epsilon\hat\omega_\epsilon=\hat\lambda$, we see that \eqref{e:pb} is equivalent to
\[
\Psi_\epsilon^*\big(\diff\hat\lambda+\mathcal L_{Z_\epsilon}\hat\omega_\epsilon\big)=\diff(h_\epsilon\diff\theta),\qquad h_\epsilon:=\partial_\epsilon H_\epsilon.
\]
Thanks to the Cartan formula, this equation can be rewritten as
\begin{equation}\label{e:c1}
\diff\Big[\Psi_\epsilon^*\big(\hat\lambda+\iota_{Z_\epsilon}\hat\omega_\epsilon\big)-h_\epsilon\diff\theta\Big]=0.
\end{equation}
We require now that $Z_\epsilon$ belong to the horizontal distribution $\mathcal H$ tangent to the level sets of $\theta$. By taking the pull-back by $(\Psi_\epsilon)^{-1}$ on both sides, we see that \eqref{e:c1} is solved if there exists a function $c_\epsilon:SU\to\R$ with
\begin{equation}\label{e:c2}
\hat\lambda+\iota_{Z_\epsilon}\hat\omega_\epsilon+\diff c_\epsilon=h_\epsilon\circ\Psi_\epsilon^{-1}\diff\theta.
\end{equation}

This equation can be decomposed in the horizontal and vertical component at $v\in SU$:
\begin{subequations}
\begin{align}
\hat\omega_\epsilon|_{\mathcal H}(Z_\epsilon,\cdot\,)&=-\hat\lambda|_{\mathcal H}-\diff|_{\mathcal H}c_\epsilon,\label{e:c3a}\\
h_\epsilon\circ\Psi_\epsilon^{-1}(v)&=\epsilon g(Z_\epsilon,v^\perp)+\partial_\theta c_\epsilon(v).\label{e:c3b}
\end{align}
\end{subequations}
For any function $c_\epsilon$, Equation \eqref{e:c3a} determines uniquely the vector field $Z_\epsilon$, and hence the isotopy $\Psi_\epsilon$, since $\hat\omega_\epsilon|_{\mathcal H}$ is a non-degenerate bilinear form for $\epsilon$ small enough. Then Equation \eqref{e:c3b} determines $h_\epsilon$, and hence the function $H_\epsilon$, uniquely.

We expand $Z_\epsilon$ and $h_\epsilon$ to the first order in $\epsilon$:
\[
Z_\epsilon=Z_0+\epsilon Z'_\epsilon,\qquad h_\epsilon=h_0+\epsilon h'_\epsilon.
\]
We choose $c_\epsilon=\epsilon^2c'_\epsilon$ for some function $c'_\epsilon$ to be determined later. Then we evaluate Equations \eqref{e:c3a} and \eqref{e:c3b} at $\epsilon=0$. From \eqref{e:c3a} we get
\[
\hat\omega_0|_{\mathcal H}(Z_0,\cdot\,)=-\hat\lambda|_{\mathcal H}.
\]
By the definition of $\hat\omega_0$ and $\hat\lambda$, this equation can be rewritten as
\[
-b_1g(\diff\pi[ Z_0^\perp],\cdot\,)=-g(v,\cdot\,)\qquad \text{on }\mathcal H,
\]
from which it follows that
\[
Z_0=-b_1^{-1}\tilde v^\perp.
\]
Second, from \eqref{e:c3b} we get that
\[
h_0=0.
\]
Dividing \eqref{e:c3a} and \eqref{e:c3b} by $\epsilon$, we arrive at the equivalent equations
\begin{subequations}
	\begin{align}
	\hat\omega_\epsilon|_{\mathcal H}(Z'_\epsilon,\cdot\,)&=b_1^{-1}\diff\hat\lambda|_{\mathcal H}(\tilde v^\perp,\cdot\,)-\epsilon\diff|_{\mathcal H}c'_\epsilon,\label{e:c4a}\\
	h'_\epsilon\circ\Psi_\epsilon^{-1}(v)&=-b_1^{-1}+\epsilon g(\diff\pi[Z'_\epsilon],v^\perp)+\epsilon\partial_\theta c'_\epsilon(v)
	.\label{e:c4b}
	\end{align}
\end{subequations}
Evaluating \eqref{e:c4a} at $\epsilon=0$ and using the expression for $\diff\hat\lambda$ obtained in Lemma \ref{l:dlambda}, we find the equation
\[
-b_1g(\diff\pi[Z'_0{}^\perp],\cdot\,)=b_1^{-1}\alpha(v^\perp)g(v^\perp,\cdot\,)-b_1^{-1}\alpha,
\]
from which it follows that 
\[
g(\diff\pi[Z'_0{}^\perp],v)=-b_1^{-2}\alpha(v),\qquad g(\diff\pi[Z'_0{}^\perp],v^\perp)=0,
\]
which is equivalent to
\[
Z_0'=b_1^{-2}\alpha(v)\tilde v^\perp.
\]
Evaluating \eqref{e:c4b} at $\epsilon=0$, we get
\[
h'_0=-b_1^{-1},
\] 
which yields the desired expansion of $H_\epsilon$ to the second order in $\epsilon$.

Let us now go to higher order under the hypothesis that $b_1$ is a non-zero constant. We write
\[
Z'_\epsilon=Z'_0+\epsilon Z_\epsilon'',\qquad h_\epsilon'=-b_1^{-1}+\epsilon h_\epsilon'',\qquad c_\epsilon'=c_0'+\epsilon c_\epsilon''
\]
and substitute these expressions in \eqref{e:c4a} and \eqref{e:c4b}. After dividing by $\epsilon$, we get the new set of equations
\begin{subequations}
	\begin{align}
	\hat\omega_\epsilon|_{\mathcal H}(Z''_\epsilon,\cdot\,)&=b_1^{-2}\alpha(v)\diff\hat\lambda|_{\mathcal H}(\tilde v^\perp,\cdot\,)-\diff|_{\mathcal H}c'_0-\epsilon\diff|_{\mathcal H}c''_\epsilon,\label{e:c5a}\\
	h''_\epsilon\circ\Psi_\epsilon^{-1}(v)&=
	b_1^{-2}\alpha(v)+\epsilon g(\diff\pi[Z_\epsilon''],v^\perp)+\partial_\theta c'_0(v)+\epsilon\partial_\theta c_\epsilon''(v).
\label{e:c5b}
	\end{align}
\end{subequations}
We now choose $c_0'(v)=b_1^{-2}\alpha(v^\perp)$, so that
\[
b_1^{-2}\alpha(v)+\partial_\theta c_0'(v)=0.
\]
In this case, evaluating \eqref{e:c5b} at $\epsilon=0$, we get $h''_0=0$. Thus we can substitute $h_\epsilon''=\epsilon h_\epsilon'''$ in \eqref{e:c5b} and dividing this equation by $\epsilon$, we find
\[
h_\epsilon'''\circ\Psi_\epsilon^{-1}(v)=g(\diff\pi[Z_\epsilon''],v^\perp)+\partial_\theta c_\epsilon''(v).
\]
Evaluating at $\epsilon=0$ we have the formula
\begin{equation}\label{e:hfinal}
h_0'''(v)=g(\diff\pi[ Z_0''],v^\perp)+\partial_\theta c_0''(v).
\end{equation}
In order to determine the first summand on the right, we evaluate \eqref{e:c5a} at $\epsilon=0$ and plug in $\tilde v$:
\[
-b_1g(\diff\pi[Z''_0{}^\perp],v)=-b_1^{-2}\Big[\alpha(v)^2+\diff\big(\alpha(e^{i\theta(v)}W^\perp)\big)\cdot v\Big].
\]
Here we have used that $b_1$ is constant. Recalling the definition of the function $f$ from \eqref{e:f}, we arrive at
\[
g(\diff\pi[Z''_0],v^\perp)=-b_1^{-3}f(v).
\]
Let $a:SU\to\R$ be the function given by Lemma \ref{l:K} and choose $c''_\epsilon=c''_0=-b_1^{-3}a$. Thanks to \eqref{e:K2} and the fact that $b_1$ is constant, we see that \eqref{e:hfinal} is equivalent to
\[
h_0'''=-b_1^{-3}f-b_1^{-3}\partial_\theta a=-\tfrac12b_1^{-3}K.
\]
Thus
\[
h_\epsilon=-\epsilon b_1^{-1}-\tfrac12\epsilon^3b_1^{-3}K
\]
and the formula for $H_\epsilon$ follows. This finishes the proof of Theorem \ref{t:nf}.\hfill\qed


\section{Trapping regions for the magnetic flow}
\label{s:kam}

We start by recalling the celebrated twist theorem of Moser \cite{Mos62}. We give here the version for flows which can be easily deduced from the original one for mappings.

Consider the standard symplectic form $\omega=\diff I\wedge\diff\varphi$ on the open annulus $(I_-,I_+)\times\T$ with coordinates $(I,\varphi)$, and let 
\begin{equation}\label{e:Haction}
H_{\epsilon,\theta}=h_0(\epsilon)+\epsilon^kh_1\circ I+o(\epsilon^k)
\end{equation}
be a Hamiltonian periodically depending on the time $\theta$. Here $h_0(\epsilon)$ is a real constant, $k$ is a positive integer and $h_1:(I_-,I_+)\to\R$ a function. We call $\Phi_{H_\epsilon}$ the Hamiltonian flow on the set $(I_-,I_+)\times\T^2$, where the second angular coordinate is given by the periodic time $\theta$.

\begin{thm}[The Moser twist theorem \cite{Mos62}]\label{t:moser}
Suppose that for some $I_0\in (I_-,I_+)$ there holds
\begin{equation*}
\frac{\diff^2 h_1}{\diff I^2}(I_0)\neq 0.
\end{equation*}
Then there exists $\epsilon_0>0$ such that for all $\epsilon\in[0,\epsilon_0]$ there is a two-dimensional torus $\mathcal T_{I_0}^\epsilon\subset (I_-,I_+)\times\T^2$ which is $C^1$-close to the torus $\{I_0\}\times \T^2$ and is invariant under $\Phi_{H_\epsilon}$. In particular, $\mathcal T_{I_0}^\epsilon$ divides $(I_-,I_+)\times\T^2$ into two open regions both invariant under the flow.
\end{thm}
\begin{rem}\label{r:moser}
Applying the Moser twist theorem to $I_0^-\in (I_-,I_0)$ and $I_0^+\in(I_0,I_+)$ arbitrarily close to $I_0$ (which can also be taken to be non-resonant), we find $\epsilon_0$ such that for $\epsilon\in[0,\epsilon_0]$
all the trajectories of $\Phi_{H_\epsilon}$ starting in the region $\mathcal U'$ between the tori $\mathcal T_{I_0^-}^\epsilon$ and $\mathcal T_{I_0^+}^\epsilon$ will remain in $\mathcal U'$ for all times.
\end{rem}
Let us see how we can apply the Moser twist theorem for local perturbations of an autonomous Hamiltonian $H_1:M\to\R$ on some symplectic surface $(M,\omega)$. Suppose that $L_{c_0}$ is a connected component of a regular level set $\{H_1=c_0\}$ which is diffeomorphic to a circle. Then there is a tubular neighborhood $U$ of $L_{c_0}$ foliated by embedded circles $c\mapsto L_c$ for $c\in(c_-,c_+)$ so that $H_1(L_c)=c$. By \cite[Section 50]{Arnold:1978aq}, there are action-angle coordinates on the neighborhood of $L_{c_0}$ foliated by the circles $L_c$. The action variable $I$ can be written as a composition $\tilde I\circ H_1$, where $\tilde I$ is the strictly monotonically increasing function given by
\[
\tilde I(c)=\int_{\mathcal A_{c_0,c}}\omega
\]
and $\mathcal A_{c_0,c}$ denotes the annular region between $L_{c_0}$ and $L_c$. We have that
\[
\frac{\diff\tilde I}{\diff c}=\int_{L_c}\eta,
\]
where $\eta$ is any one-form on $U$ satisfying $\omega=\eta\wedge\diff H_1$. For example we can take 
\begin{equation}\label{e:eta}
\eta=\omega\big(\tfrac{1}{|\nabla H_1|^2}\nabla H_1,\cdot\,\big),
\end{equation}
where the gradient and the norm are taken with respect to an arbitrary metric. With this choice, we see that there exists a positive constant $\delta>0$ not depending on $c$ such that
\begin{equation}\label{e:ds}
\frac{\diff \tilde I}{\diff c}\geq \delta\int_{L_c}|\nabla H_1|^{-1}\diff s,
\end{equation}
where $\diff s$ is the arc-length of $L_c$.

If $\tilde h_1$ is the inverse function of $\tilde I$, we conclude that $H_1=\tilde h_1\circ \tilde I$ and
\[
\forall\,e=\tilde I(c),\qquad \frac{\diff^2 \tilde h_1}{\diff e^2}(e)\neq0\quad\Longleftrightarrow\quad\frac{\diff^2\tilde I}{\diff c^2}(c)\neq0.
\]
\begin{dfn}
We say that $L_c$ is \textit{non-resonant} if $\frac{\diff^2\tilde I}{\diff c^2}(c)\neq0$.
\end{dfn}
\begin{rem}\label{r:idc}
Two simple conditions to show the existence of non-resonant circles are as follows. First, if $\frac{\diff \tilde I}{\diff c}$ is not constant on $(c_-,c_+)$, then there exists a non-resonant circle $L_c$ with $c\in(c_-,c_+)$. Second, if
\[
\lim_{c\to c_+}\frac{\diff\tilde I}{\diff c}=+\infty,
\]
then there is a sequence $(c_n)\subset(c_-,c_+)$ with the property that $c_n\to c_+$ and $L_{c_n}$ is a non-resonant circle. A similar statement holds for $c_-$ instead of $c_+$.
\end{rem}
\begin{rem}\label{r:scale}
The non-resonance condition for a circle $L_c$ is a property that is invariant upon multiplying the symplectic form by a non-zero constant. In view of the applications to magnetic fields, we need to consider the case in which $\omega$ also depends on $H_1$ in a certain way. More specifically, suppose that $H_1$ is positive and that $\omega=H_1^{-1}\omega_1$ for some symplectic form $\omega_1$. Then
\[
H_1^{-1}\omega_1=\eta\wedge\diff H_1\quad\Longleftrightarrow\quad \omega_1=\eta\wedge\diff(\tfrac12H_1^2).
\]
This means that
\[
\frac{\diff\tilde I_{\omega}}{\diff c}(c)=\frac{\diff\tilde I_{\omega_1}}{\diff c_1}(\tfrac 12c^2),
\]
where we put the symplectic form in subscript to distinguish the two cases. Differentiating in $c$, we get
\[
\frac{\diff^2\tilde I_{\omega}}{\diff c^2}(c)=c\frac{\diff^2\tilde I_{\omega_1}}{\diff^2 c_1}(\tfrac 12c^2),
\]
so that a level set of $H_1$ is non-resonant with respect to $H^{-1}_1\omega_1$ if and only if it is non-resonant as a level set of $\tfrac{1}{2}H_1^2$ with respect to $\omega_1$.
\end{rem}
Let us consider a family of time-depending Hamiltonians on $U$ having the form
\[
H_{\epsilon,\theta}=h_0(\epsilon)+\epsilon^kH_1+o(\epsilon^k),
\]
where $h_0(\epsilon)$ is some constant, $k$ is a positive integer and the dependence in the time $\theta\in\T$ is periodic. Using the action-angle coordinates on $U$ around $L_{c_0}$ described above, we can bring the Hamiltonian in the form \eqref{e:Haction}. If now $L_{c_0}$ is a non-resonant circle, the following statement follows directly from the Moser twist Theorem \ref{t:moser} and the subsequent Remark \ref{r:moser}: for all neighborhoods $\mathcal U'$ of $L_{c_0}$, there exists an $\epsilon_0>0$ and a neighborhood $\mathcal U''\subset \mathcal U'$ of $L_{c_0}$ with the property that for all $\epsilon\in[0,\epsilon_0]$, every trajectory of $\Phi_{H_\epsilon}$ starting in $\mathcal U''\times\T$ will stay in $\mathcal U'\times \T$ for all times.

We can now use the above discussion to prove trapping for strong magnetic fields.
\begin{proof}[Proof of Theorem \ref{t:kam}]
Suppose that $L=L_{c_0}$ is a non-resonant circle for $\zeta=b_1^{-2}$ with respect to the symplectic form $\mu$ on $M$. By Remark \ref{r:scale}, $L$ is a non-resonant circle for $H_1:=-\tfrac12b_1^{-1}$ with respect to $\omega=b_1\mu$. Let $U$ be an arbitrary neighborhood of $L$. Upon shrinking it, we can assume that $U$ is a tubular neighborhood of $L$ so that $SU$ admits an angular function $\theta$. We apply Theorem \ref{t:nf} to $U$ and a further tubular neighborhood $U'$ of $L$ with $\bar U'\subset U$: There exists $\epsilon_0$ such that for all $\epsilon\in(0,\epsilon_0]$ there is a Hamiltonian $H_{\epsilon,\theta}=h_0(\epsilon)+\epsilon^2H_1+o(\epsilon^2)$ with $h_0(\epsilon)=0$ such that the Hamiltonian flow $\Phi_{H_\epsilon}$ is conjugated via the map $\Psi_\epsilon$ to the magnetic flow $\Phi_{(g,\epsilon^{-1}b_1)}$, up to time reparametrization. Applying the Moser twist theorem, we see that, upon shrinking $\epsilon_0$, there is a neighborhood $U''$ of $L$ in $U'$ such that for all $\epsilon\in(0,\epsilon_0)$ a trajectory of $\Phi_{H_\epsilon}$ starting in $SU''$ will stay in $SU'$ for all times. This means that all $(g,\epsilon^{-1}b_1)$-geodesics with initial velocity vector in $\Psi_\epsilon(S U'')$ will stay in $\pi(\Psi_\epsilon(SU'))$ for all times. Upon shrinking $\epsilon_0$ we see that there exists a neighborhood $U'''$ of $L$ with $SU'''\subset\Psi_\epsilon(SU'')$ and that $\pi(\Psi_\epsilon(SU'))\subset U$. This means that any $(g,\epsilon^{-1}b_1)$-geodesics starting in $U_1:=U'''$ will stay inside $U$ for all times. 

This finishes the proof of Theorem \ref{t:kam} for the function $\zeta=b_1^{-2}$. The proof of Theorem \ref{t:kam} when $b_1$ is a positive constant and the function $\zeta$ is the Gaussian curvature $K$ is analogous and is left to the reader.
\end{proof}
We now give some criteria for the existence of non-resonant circles for $H_1$. The first criterion is due to Castilho \cite[Corollary 1.3]{Castilho:2001}. We present the short proof here for the convenience of the reader.
\begin{lem}\label{l:mincircle}
Suppose that $H_1$ attains a minimum or maximum at $L_{c_0}$ and that there are no other critical points of $H_1$ in a neighborhood of $L$ (for example $L_{c_0}$ is a Morse--Bott component for $H_1$). Then there are non-resonant circles on either sides of $L_{c_0}$ accumulating at $L_{c_0}$.
\end{lem}
\begin{proof}
Let us suppose without loss of generality that $c_0$ is the minimum of $H_1$. Consider a family of circles $c\mapsto L_c$ contained in regular level sets of $H_1$ staying on one side of $L_{c_0}$ for $c\in(c_0,c_1)$. In particular, $L_c\to L_{c_0}$ uniformly as $c\to c_0$. It is enough to show that
\begin{equation}\label{e:mincircle}
\lim_{c\to c_0}\frac{\diff \tilde I}{\diff c}=+\infty.
\end{equation}
By inequality \eqref{e:ds}, we get the lower bound
\[
\frac{\diff \tilde I}{\diff c}=\delta\int_{L_c}|\nabla H_1|^{-1}\diff s\geq \delta(\max_{L_c}|\nabla H_1|)^{-1}\mathrm{length}(L_c).
\]
As $c$ tends to $c_0$, $\mathrm{length}(L_c)$ is bounded away from zero, while $\max_{L_c} |\nabla H_1|$ converges to zero. The limit in \eqref{e:mincircle} follows.
\end{proof}
\begin{proof}[Proof of Corollary \ref{c:mincircle}]
We recall that by Remark \ref{r:scale}, a circle $L$ is non-resonant for $H_1:=-\tfrac12b_1^{-1}$ with respect to $\omega=b_1\mu$ if and only if it is non-resonant for $\zeta=b_1^{-2}$ with respect to $\mu$. Let $U$ be an arbitrary neighborhood of the critical circle $L_{c_0}$. We can assume that $U$ is a tubular neighborhood of $L$. We apply Theorem \ref{t:nf} to $U$ and a neighborhood $U'$ of $L$ such that $\bar U'\subset U$ and get a diffeomorphism $\Psi_\epsilon$ and a function $H_{\epsilon,\theta}=h_0(\epsilon)+\epsilon^2 H_1+o(\epsilon^2)$ with $h_0(\epsilon)=0$ for all $\epsilon\in[0,\epsilon_0]$. By Lemma \ref{l:mincircle} there are non-resonant circles $L_-$ and $L_+$ for $H_1$ inside $U'$ on either side of $L_{c_0}$. Applying the Moser twist Theorem \ref{t:moser} to $L_-$ and $L_+$ we see that, for a possibly smaller $\epsilon_0$, there is a neighborhood $U''$ of $L$ such that for all $\epsilon\in[0,\epsilon_0]$ every trajectory of $\Phi_{H_\epsilon}$ with initial condition in $SU''$ remains in $SU'$ for all times. Taking an even smaller $\epsilon_0$, we get $\pi(\Psi_\epsilon(SU'))\subset U$ and we find a neighborhood $U'''$ of $L$ such that $SU'''\subset \Psi_\epsilon(SU'')$. Thus any $(g,\epsilon^{-1}b_1)$-geodesics starting in $U_1:=U'''$ will be contained in $U$ for all times.

The proof of Corollary \ref{c:mincircle} for $\zeta=b_1^{-2}$ is completed. For $\zeta=K$ and $b_1$ a positive constant, the proof is analogous and we omit it.
\end{proof}
Let us analyze the situation in a neighborhood of an isolated minimum or maximum $q_*\in U$ for $H_1$. Let $c_0:=H_1(q_*)$, so that we have a family of circles $L_c$ for $c$ in some interval $(c_0,c_+)$ (or $(c_-,c_0)$) converging uniformly to $q_*$ as $c\to c_0$.
\begin{lem}\label{l:minpoint}
Let $q_*\in M$ be an isolated local minimum or maximum for $H_1$. If $q_*$ is degenerate, then
\begin{equation}\label{e:qdeg}
\lim_{c\to c_0}\frac{\diff \tilde I}{\diff c}=+\infty.
\end{equation}
If $q_*$ is non-degenerate, then
\begin{equation}\label{e:qnondeg}
\lim_{c\to c_0}\frac{\diff \tilde I}{\diff c}= \frac{2\pi}{\sqrt{\det\mathrm{Hess}\, H_1 (q_*)}} ,\qquad \lim_{c\to c_0}\frac{\diff^2 \tilde I}{\diff c^2}=\pi\Delta\rho(q_*),
\end{equation}
where the Hessian is taken in Darboux coordinates around $q_*$, and $\rho$ is a function with the property that $\omega=\rho\,\diff x\wedge\diff y$ in coordinates $(x,y)$ around $q_*$ such that $H_1=c_0+\tfrac{1}{2}(x^2+y^2)$. 
\end{lem}
\begin{proof}
We deal only with the case of a local minimum and leave the case of a local maximum to the reader. We make the preliminary remark that, up to an additive constant, we have
\[
\tilde I(c)=\int_{\{H_1\leq c\}}\omega,
\]
where $\{H_1\leq c\}$ is a sublevel set of $H_1$ in a neighborhood of $q_*$.

Assume first that $q_*$ is degenerate. There exists Darboux coordinates $(x,y)$ around $q_*$ such that $H_1=c_0+\tfrac12(ux)^2+o(r^2)$ for some constant $u\geq 0$. In particular, 
we find a constant $C>0$ such that
\[
H_1-c_0\leq \tfrac12(ux)^2+\tfrac12(Cr)^3. 
\]
We distinguish two cases. 

If $u=0$, then we readily see that
\[
\{r^2\leq C^{-2}2^{2/3}(c-c_0)^{2/3}\}\subset\{H_1\leq c\}.
\]
Therefore,
\begin{align*}
\frac{\tilde I(c)}{c-c_0} &=\frac{1}{c-c_0}\int_{\{H_1\leq c\}}\omega \\
				&\geq \frac{1}{c-c_0} \int_{\{r^2\leq C^{-2}2^{2/3}(c-c_0)^{2/3}\}}\omega \\
				& = \frac{\pi C^{-2}2^{2/3}(c-c_0)^{2/3}}{c-c_0}\\
				&=\frac{\pi 2^{2/3}}{C^2(c-c_0)^{1/3}}.
\end{align*}
Since the last term diverges as $c$ tends to $c_0$, we get
\eqref{e:qdeg}. 

If $u\neq0$, then 
\[
\Big\{x^2\leq u^{-2}(c-c_0),\ \ y^2\leq C^{-2}(c-c_0)^{2/3}-u^{-2}(c-c_0)\Big\}\subset\{H_1\leq c\}.
\]
Therefore, arguing as above we obtain that 
\[
\frac{\tilde I(c)}{c-c_0}\geq \frac{ u^{-1}(c-c_0)^{1/2}\big(C^{-2}(c-c_0)^{2/3}-u^{-2}(c-c_0)\big)^{1/2}}{c-c_0}=\frac{1}{uC(c-c_0)^{1/6}}\Big(1-\frac{C^2(c-c_0)^{1/3}}{u^2}\Big)^{1/2}.
\]
Since the rightmost term diverges as $c$ tends to $c_0$, we again get \eqref{e:qdeg}.

Let us assume now that $q_*$ is non-degenerate. There are Darboux coordinates centered at $q_*$ such that $H_1=c_0+\tfrac12ur^2+o(r^2)$, where $u:=\sqrt{\det\mathrm{Hess}\, H_1 (q_*)}$. Then $\nabla H_1=ur\partial_r+o(r)$ and
\[
\frac{\nabla H_1}{|\nabla H_1|^2}=\frac{1}{ur}\partial_r+Y,\qquad \text{where}\quad \lim_{r\to 0}r|Y|=0.
\]
We compute
\[
\frac{\diff\tilde I}{\diff c}=\int_{L_c}\iota_{\frac{\nabla H_1}{|\nabla H_1|^2}}\omega=\int_{L_c}\Big(\iota_{\frac{1}{ur}\partial_r}\omega+\iota_{Y}\omega\Big)=\int_{L_c}\Big(\frac{\diff\theta}{u}+\iota_Y\omega\Big)=\frac{2\pi}{u}+\int_{L_c}\iota_Y\omega.
\]
If $\diff s$ is the arc-length of $L_c$, then there exists $C_1>0$ such that $\iota_Y\omega|_{L_c}\leq C_1\diff s$. From the expansion of $H_1$ at $q_*$ we also see that there exists $C_2>0$ such that $\mathrm{length}(L_c)\leq C_2\max_{L_c}r$. Thus
\[
\lim_{c\to c_0}\Big|\int_{L_c}\iota_Y\omega\Big|\leq \lim_{c\to c_0}\max_{L_c}|Y|C_1\mathrm{length}(L_c)\leq\lim_{c\to c_0}\max_{L_c}|Y|C_1C_2\max_{L_c}r=C_1C_2\lim_{r\to0}r|Y|=0
\]
and the first limit in \eqref{e:qnondeg} follows. For the second limit, we notice that in coordinates where $H_1=c_0+\tfrac12r^2$, we have $\eta=\rho \diff\theta$. For $c>c_0$ we define $r_c$ via the equation $c=c_0+\tfrac12r_c^2$. Thanks to Stokes' theorem, we have
\[
\frac{\diff^2\tilde I}{\diff c^2}=\int_{L_c}\frac{1}{r}\iota_{\partial_r}\diff\eta=\int_{0}^{2\pi}\frac{\partial_r\rho(r_c,\theta)}{r_c}\diff\theta.
\]
Since $\partial_r\rho(0,\theta)=\diff\rho\cdot e^{i\theta}$, we have $\int_{0}^{2\pi}\partial_r\rho(0,\theta)\diff\theta=0$, and hence
\begin{align*}
\lim_{c\to c_0}\frac{\diff^2\tilde I}{\diff c^2}& =\lim_{c\to c_0}\int_{0}^{2\pi}\frac{\partial_r\rho(r_c,\theta)-\partial_r\rho(0,\theta)}{r_c}\diff\theta\\
							& =\int_{0}^{2\pi}\partial^2_{rr}\rho(0,\theta)\diff\theta\\
							&=\int_{0}^{2\pi}\mathrm{Hess}\, \rho (q_*)[e^{i\theta},e^{i\theta}]\diff\theta\\
&=\pi \mathrm{Trace}(\mathrm{Hess}\, \rho(q_*))\\
&=\pi\Delta\rho(q_*).\qedhere
\end{align*}
\end{proof}
From the lemma above we deduce immediately the following result which, together with Remark \ref{r:idc}, yields Theorem \ref{t:S2}.
\begin{lem}
Suppose that $M=S^2$ and that $H_1$ has exactly one non-degenerate minimum at $q_{\min}\in S^2$ and one non-degenerate maximum at $q_{\max}\in S^2$ so that there exists a family of embedded circles $c\mapsto L_c$, $c\in(H_1(q_{\min}),H_1(q_{\max}))$. Assume that
\[
\sqrt{\det \mathrm{Hess}\, \zeta ({q_{\min}})}\neq\sqrt{\det \mathrm{Hess}\, \zeta ({q_{\max}})},
\]
where the Hessian is taken in Darboux coordinates for $\mu$. Then 
$$\frac{\diff \tilde I}{\diff c}:(H_1(q_{\min}),H_1(q_{\max}))\to\R$$
is not constant.\hfill\qed 
\end{lem}
\begin{proof}[Proof of Corollary \ref{c:minpoint}]
Thanks to Remark \ref{r:scale}, a circle $L$ is non-resonant for $H_1:=-\tfrac12b_1^{-1}$ with respect to $\omega=b_1\mu$ if and only if it is non-resonant for $\zeta=b_1^{-2}$ with respect to $\mu$. Let $U$ be an arbitrary neighborhood of $q_*$ which we can suppose to be diffeomorphic to a ball. We apply Theorem \ref{t:nf} to $U$ and a neighborhood $U'$ of $q_*$ such that $\bar U'\subset U$. We obtain a diffeomorphism $\Psi_\epsilon$ and a function $H_{\epsilon,\theta}=h_0(\epsilon)+\epsilon^2H_1+o(\epsilon^2)$ with $h_0(\epsilon)=0$ for every $\epsilon\in[0,\epsilon_0]$. By Lemma \ref{l:minpoint} there is a non-resonant circle $L$ for $H_1$ inside $U'$. Using the Moser twist theorem \ref{t:moser} on $L$, we can find, by choosing a smaller $\epsilon_0$, a neighborhood $U''$ of $q_*$ with the property that, for all $\epsilon\in[0,\epsilon_0]$, every flow line of $\Phi_{H_\epsilon}$ starting in $SU''$ stays in $SU'$ forever. After shrinking $\epsilon_0$ again, we have $\pi(\Psi_\epsilon(SU'))\subset U$ and there is a neighborhood $U'''$ of $q_*$ such that $SU'''\subset\Psi_\epsilon(SU'')$. We deduce that any $(g,\epsilon^{-1}b_1)$-geodesics passing through $U_1:=U'''$ will stay in $U$ forever. We have thus showed Corollary \ref{c:minpoint} for $\zeta=b_1^{-1}$. The proof for $\zeta=K$ when $b_1$ is a positive constant is completely analogous.
\end{proof}
The last criterion deals with a saddle point of $H_1$ and together with Remark \ref{r:idc} yields Theorem \ref{t:saddle}.
\begin{lem}\label{l:saddle}
	Suppose that $H_1$ has a non-degenerate saddle critical point $q_*$ such that there are no critical values of $H_1$ in the interval $(c_0,c_1)$ for $c_0:=H_1(q_*)$ and some $c_1>c_0$. Then there exists a family of circles $c\mapsto L_c$, $c\in(c_0,c_1)$ such that
	\[
	\lim_{c\to c_0}\mathrm{dist}(q_*,L_c)=0,\qquad  \lim_{c\to c_0}\frac{\diff\tilde I}{\diff c}=+\infty.
	\]
\end{lem}
\begin{proof}
There is a chart $U'\to(-\delta,\delta)^2$ centered at $q_*$ such that $H_1$ has the form $H_1=c_0+xy$ in the corresponding coordinates. Since there are no critical values in $(c_0,c_1)$, there is a smooth family of circles $c\mapsto L_c$, $c\in(c_0,c_1)$ with the property that 
$$H_1(L_c)=c,\quad \text{and}\ L_c\cap U'=\{(c-c_0)\delta^{-1}<x<\delta,\ y=(c-c_0)x^{-1}\}.$$ 
We compute
\begin{align*}
\frac{\diff\tilde I}{\diff c}&\geq \delta\int_{L_c}|\nabla H_1|^{-1}\diff s\\
				&\geq \delta\int_{L_c\cap U'}|\nabla H_1|^{-1}\diff s\\
				&=\delta\int_{(c-c_0)\delta^{-1}}^\delta r^{-1}\sqrt{1+\frac{(c-c_0)^2}{x^4}}\diff x\\
				&\geq \delta\int_{(c-c_0)^{1/2}}^\delta r^{-1}\diff x\\
				&\geq \delta\int_{(c-c_0)^{1/2}}^\delta \frac{1}{2x}\diff x,
\end{align*}
where in the last inequality we used that $y=(c-c_0)x^{-1}\leq x$ if $x\geq\sqrt{c-c_0}$. It is now easy to see that the last integral diverges for $c$ tending to $c_0$. 
\end{proof}


\section{Zoll-rigidity for strong magnetic fields}
\label{s:zoll1}
In this section we give a proof of Theorem \ref{t:Zoll1} on the Zoll-rigidity for sets $R\subset(0,\infty)$ accumulating to $0$.
By definition of geodesic curvature, $(g,\epsilon^{-1}b_1)$-geodesics and $(g,-\epsilon^{-1}b_1)$-geodesics are the same curves parametrised in opposite directions. Thus we can assume that
\[
\int_Mb_1\mu\geq0.
\]
Let us first suppose that $b_1:M\to\R$ is not constant. Then $\max b_1>0$ and we fix positive numbers $\delta_1<\delta_2<\delta_3$ in the interval $(\min b_1,\max b_1)$. We define $U_i:=\{b_1>\delta_i\}$ for $i=1,2$ and let $U_3$ be a connected component of $\{b_1>\delta_3\}$ such that
\begin{equation*}\label{e:maxb}
U_3\cap\{b_1=\max b_1\}\neq\varnothing.
\end{equation*}
Since $\bar U_1\neq M$, there exists a section $W$ of $\pi:SU_1\to U_1$. Thus we can apply Theorem \ref{t:nf} with $U=U_1$, $U'=U_2$ which means that, up to shrinking $\epsilon_0$, we can find an isotopy of embeddings $\Psi_\epsilon:SU'\to SU$ for $\epsilon\in[0,\epsilon_0]$ such that
\[
\Psi_\epsilon^*\hat\omega_\epsilon=-\pi^*(b_1\mu)+\diff(H_\epsilon\diff\theta),
\]
where $\theta$ is the angular function of $W$ and $H_\epsilon:SU\to\R$ satisfies
\[
H_\epsilon=\epsilon^2H'_\epsilon,\qquad H'_0=-(2b_1)^{-1}\circ\pi.
\]
Let us remember the definition of the vector fields $X_{H_{\epsilon,\theta}}$ and $X_{H_\epsilon}$ in \eqref{e:X1}, \eqref{e:X2} and of the associated flows $\varphi_{H_\epsilon}$ and $\Phi_{H_\epsilon}$. We have
\begin{equation}\label{e:phiphi}
X_{H_{\epsilon,\theta}}=\epsilon^2X_{H'_{\epsilon,\theta}}\quad\forall\,\theta\in \T,\qquad \varphi_{H_\epsilon}^{\theta_0,\theta_1}=\varphi_{H'_\epsilon}^{\theta_0,\epsilon^2\theta_1},\qquad\forall\,\theta_0\in \T,\ \theta_1\in \R.
\end{equation}
Since $\bar U_3\subset U_2$, there exists $\ell_0>0$ such that if $q:[0,1]\to M$ is a curve with $q(0)\in U_3$ and length less than $\ell_0$, then $q([0,1])\subset U_2$. Up to shrinking $\epsilon_0$, there exists a constant $\tau_0>0$ such that $\varphi_{H'_\epsilon}^{\theta_0,\tau_1}:U_3\to U_2$ is well-defined for every $\tau_1\in[0,\tau_0]$ and every $\epsilon\in[0,\epsilon_0]$. By \eqref{e:phiphi} this implies that for every $\epsilon\in[0,\epsilon_0]$, the flow $\varphi_{H_\epsilon}^{\theta_0,\theta_1}:U_3\to U_2$ is well-defined for all $\theta_1\in[0,\epsilon^{-2}\tau_0]$ and, hence the flow $\Phi_{H_\epsilon}^{\theta_1}:SU_3\to SU_2$ is also well-defined for $\theta_1\in[0,\epsilon^{-2}\tau_0]$. Up to shrinking $\epsilon_0$, we assume that
\[
2\pi<\epsilon^{-2}_0\tau_0
\]
so that the set
\[
Q_\epsilon:=\{q\in U_3\ |\ \varphi_{H_\epsilon}^{0,2\pi}(q)=q\}
\]
is well-defined.

The set $Q_\epsilon$ plays a central role in the proof of Theorem \ref{t:Zoll1} and now we want to study it better. To this purpose, let us consider the following construction. Let $z:[0,T]\to SU_1$ be a curve. We define $\tilde\theta_z:[0,T]\to\R$ to be any real lift of the function $\theta\circ z:[0,T]\to \T$. We define
\[
\Delta\theta_z:=\tilde\theta_z(T)-\tilde\theta_z(0)=\int_0^T\diff\theta\big(\tfrac{\diff}{\diff t}z\big)\diff t.
\]
In particular, if $z(T)=z(0)$, we see that $\Delta\theta_z=2k\pi$ for some integer $k$. Let us specialize this construction further and consider for any $v\in SU_1$ the $(g,\epsilon^{-1}b_1)$-geodesic $\gamma_v:\R\to M$ with $\dot\gamma(0)=v$. We define the set
\[
S:=\{(v,\ell)\in SU_1\times (0,\infty)\ |\ \gamma_v([0,\ell])\subset U_1\}.
\]
Then the function
\begin{equation}\label{e:deltatheta}
\Delta\theta:S\to\R,\qquad \Delta\theta(v,\ell):=\Delta\theta_{\dot\gamma_v|_{[0,\ell]}}
\end{equation}
is well-defined and continuous since the function $\diff\theta\big(\tfrac{\diff}{\diff t}\dot\gamma_v\big)$ is continuous in $v$.
\begin{lem}\label{l:Q}
A point $q\in U_3$ belongs to $Q_\epsilon$ if and only if the $(g,\epsilon^{-1}b_1)$-geodesics with initial velocity vector $\Psi_\epsilon(W(q))$ has some period $\ell$, it is contained in $\Psi_\epsilon(SU_2)$ and $\Delta\theta(\Psi_\epsilon(W(q)),\ell)=2\pi$.
\end{lem}
\begin{proof}
Let us take $q\in U_3$ and consider $\eta_{W(q)}$ to be the maximal solution of $\Phi_{H_\epsilon}$ with $\eta_{W(q)}(0)=W(q)$ and $\gamma_{\Psi_\epsilon(W(q))}$ to be the $(g,\epsilon^{-1}b_1)$-geodesic with $\dot\gamma_{\Psi_\epsilon(W(q))}(0)=\Psi_\epsilon(W(q))$, which are related by Equation \eqref{e:gammaeta} and the time reparametrization $t:[0,\theta_\infty)\to[0,\ell_\infty)$. Let $\ell\in[0,\ell_\infty)$ and $\theta\in[0,\theta_\infty)$ be such that $t(\theta)=\ell$. Then $\dot\gamma(\ell)=\dot\gamma(0)$ if and only if $\eta_{W(q)}(\theta)=\eta_{W(q)}(0)$. In this case, $\Delta\theta(\Psi_\epsilon(W(q)),\ell)=2k\pi$ for some $k\in\Z$. By \eqref{e:gammaeta}, it follows that
\begin{equation}\label{e:2kpi}
2k\pi=\Delta\theta(\Psi_\epsilon(W(q)),\ell)=\Delta\theta_{\Psi_\epsilon(\eta_{W(q)})|_{[0,\theta]}}=\Delta\theta_{\eta_{W(q)}|_{[0,\theta]}}=\theta,
\end{equation}
where we have used the fact that $\Psi_\epsilon$ is isotopic to the standard inclusion and that \eqref{e:theta+} holds. It is now immediate to check the equivalence in the statement of the lemma.
\end{proof}
We now show that $Q_\epsilon$ is not empty for $\epsilon$ small enough. This argument goes back to \cite{Ginzburg:1987lq} and we sketch here how to adapt it to our context.
\begin{lem}\label{l:gin}
Let $b_1$ be non-constant and let $U_1,U_2,U_3$ be chosen as above. Up to shrinking $\epsilon_0$, the set $Q_\epsilon$ is not empty for all $\epsilon\in[0,\epsilon_0]$.
\end{lem}
\begin{proof}
Let $\epsilon\in[0,\epsilon_0]$. For every $q\in U_3$ let us define the curve 
\[
\bar\eta_q:[0,2\pi]\to U_2,\qquad \bar\eta_q(\theta)=\varphi_{\epsilon}^{0,\theta}(q),\qquad\forall\,\theta\in[0,2\pi].
\]
Up to shrinking $\epsilon_0$, there is a function $u_q:[0,2\pi]\to T_qM$ such that $s\mapsto\exp_q(su_q(\theta))$, $s\in[0,1]$ is the unique minimizing geodesic between $q$ and $\bar\eta_q(\theta)$. We define the disc
\[
\Gamma_q:[0,1]\times [0,2\pi]\to M,\qquad \Gamma_q(s,\theta):=\exp_q(su_q(\theta)),\qquad\forall\,(s,\theta)\in[0,1]\times[0,2\pi].
\]
Let us consider the following functional
\[
\mathbb A_\epsilon:U_3\to \R,\qquad -\int_{D^2}\Gamma_q^*(b_1\mu)+\int_0^{2\pi}H_{\epsilon,\theta}(\bar\eta_q(\theta))\diff\theta.
\]
Since $\bar\eta_q$ satisfies $b_1\mu(\dot{\bar\eta}_q,\cdot\,)=-\diff H_{\epsilon,\theta}$, we obtain
\[
\diff_q\mathbb A_\epsilon\cdot\xi=\int_0^1(b_1\mu)_{\Gamma_q(s,2\pi)}\big(\tilde\xi(s),\partial_s\Gamma_q(s,2\pi)\big)\diff s,\qquad\forall\,\xi\in T_qM,
\]
where $\tilde\xi:[0,1]\to TM$ is the unique Jacobi field along the geodesic $s\mapsto \Gamma_q(s,2\pi)$ such that $\tilde\xi(0)=\xi$ and $\tilde\xi(1)=\diff_q\varphi^{0,2\pi}_\epsilon[\xi]$. We claim that $q\in Q_\epsilon$ if and only if $\diff_q\mathbb A_\epsilon=0$. We know that $q\in Q_\epsilon$, if and only if $s\mapsto\Gamma_q(s,2\pi)=q$ is constant. Thus we see that $q\in Q_\epsilon\Rightarrow \diff_q\mathbb A_\epsilon=0$. If $q\notin Q_\epsilon$, then $\partial_s\Gamma_q(0,2\pi)$ is a non-zero element of $T_qM$ and we can find $\xi\in T_qM$ such that $b_1\mu(\xi,\partial_s\Gamma_q(0,2\pi))>0$. Since $\diff\varphi_\epsilon^{0,2\pi}[\xi]$ and $\xi$ are close, we also have that $(b_1\mu)_{\Gamma_q(s,2\pi)}(\tilde\xi(s),\partial_s\Gamma_q(s,2\pi))>0$ for all $s\in[0,1]$. Thus $\diff_q\mathbb A_\epsilon\cdot\xi>0$.

We now show that $\mathbb A_\epsilon$ has a critical point by checking that $\mathbb A_\epsilon$ has an interior maximum. Indeed, by \eqref{e:phiphi}, there exists $C>0$ such that
\begin{equation}\label{e:dotbareta}
|\dot{\bar\eta}_q|\leq C\epsilon^2,
\end{equation}
which also implies that
\begin{equation}\label{e:dist}
|\mathrm{dist}(q,\bar\eta_q(\theta))|\leq 2\pi C\epsilon^2,\qquad\forall\,\theta\in[0,2\pi].
\end{equation}
Estimates \eqref{e:dotbareta} and \eqref{e:dist} together yield $|\diff\Gamma_q|\leq C'\epsilon^2$ for some constant $C'>0$, so that
\[
-\int_{D^2}\Gamma_q^*(b_1\mu)=o(\epsilon^2).
\]
On the other hand, using again \eqref{e:dist}, we see that
\[
\int_0^{2\pi}H_{\epsilon,\theta}(\bar\eta_q(\theta))\diff\theta=\int_0^{2\pi}H_{0,\theta}(q)\diff\theta+o(\epsilon^2)=-\epsilon^22\pi (2b_1)^{-1}(q)+o(\epsilon^2).
\]
Putting the two estimates together, we get
\begin{equation}\label{e:Aepsilon}
\mathbb A_\epsilon=\epsilon^2\big[-\pi b_1^{-1}+o(1)\big].
\end{equation}
By definition of $U_3$, the function $-\pi b_1^{-1}$ attains the maximum $-\pi(\max b_1)^{-1}$ in the interior. By \eqref{e:Aepsilon}, the same is true for $\mathbb A_\epsilon$.
\end{proof}
\begin{lem}\label{l:U3minusQ}
Let $b_1$ be non-constant and let $U_1,U_2,U_3$ be chosen as above. Up to shrinking $\epsilon_0$, the set $U_3\setminus Q_\epsilon$ is not empty for all $\epsilon\in(0,\epsilon_0]$. 
\end{lem}
\begin{proof}
Let $\delta_4$ be a regular value of $b_1$ in the interval $(\delta_3,\max b_1)$. By the definition of $U_3$, there exists $q\in U_3$ with $b_1(q)=\delta_4$. Then we can find $\tau_1>0$ such that the curve
\[
q'_0:[0,\tau_1]\to U_3,\qquad q'_0(\tau):=\varphi_{H'_0}^{0,\tau}(q)
\]
is an embedding since $H'_0=-b_1^{-1}\circ\pi$ and therefore $q'_0$ parametrizes a piece of the regular level set $\{b_1=\delta_4\}$. Up to shrinking $\epsilon_0$, all curves
\[
q'_\epsilon:[0,\tau_0]\to U_3\qquad q'_\epsilon(\theta):=\varphi_{H'_\epsilon}^{0,\tau}(q)
\]
are embedded because $q'_\epsilon\to q'_0$ in the $C^1$-topology as $\epsilon\to 0$.
By \eqref{e:phiphi}, there holds 
\[
q'_\epsilon(\epsilon^2 \theta)=\varphi_{H_\epsilon}^{0,\theta}(q),
\]
so that $\theta\mapsto \varphi_{H_\epsilon}^{0,\theta}(q)$ is an embedding for $\theta\in[0,\epsilon^{-2}\tau_1]$. Then $q\notin Q_\epsilon$, if $\epsilon\in(0,\epsilon_0]$ and we shrink $\epsilon_0$ so that $2\pi\leq \epsilon^{-2}_0\tau_1$.
\end{proof}
\begin{lem}\label{l:U3equalQ}
Let $b_1$ be non-constant and let $U_1,U_2,U_3$ be chosen as above. Up to shrinking $\epsilon_0$, for all $\epsilon\in(0,\epsilon_0]$ the following property holds: if the system $(g,\epsilon^{-1}b_1)$ is Zoll, then $Q_\epsilon=U_3$.
\end{lem}
\begin{proof}
The set $Q_\epsilon$ is closed. Moreover, it is not empty by Lemma \ref{l:gin}. Thus by the connectedness of $U_3$, the lemma follows by showing that $Q_\epsilon$ is open, whenever $(g,\epsilon^{-1}b_1)$ is Zoll. Since $(g,\epsilon^{-1}b_1)$ is Zoll, the function $\ell:U_3\to\R$ associating to each $q\in U_3$, the period of the $(g,\epsilon^{-1}b_1)$-geodesic $\dot\gamma_{\Psi_\epsilon(W(q))}$ with initial velocity $\Psi_\epsilon(W(q))$ is continuous. Let now $q_0\in Q_\epsilon$. By Lemma \ref{l:Q}, there holds $\dot\gamma_{\Psi_\epsilon(W(q_0))}|_{[0,\ell(q_0)]}$  and $\Delta\theta(\Psi_\epsilon(W(q_0)),\ell(q_0))=2\pi$. By the continuity of $\ell$, we can also find an open neighborhood $B$ of $q_0$ inside $U_3$ such that
\[
q\in B\quad\Rightarrow\quad \dot\gamma_{\Psi_\epsilon(W(q))}\subset \Psi_\epsilon(SU_2)\subset SU_1.
\]
Thus the map $q\mapsto\Delta\theta(\Psi_\epsilon(W(q)),\ell(q))$ is well-defined on $B$, continuous and integer valued. This implies that
\[
\Delta\theta(\Psi_\epsilon(W_\epsilon(q)),\ell(q))=2\pi\qquad\forall\,q\in B.
\]
By Lemma \ref{l:Q}, this implies that $B\subset Q_\epsilon$. Since $q_0\in Q_\epsilon$ was arbitrary, we see that $Q_\epsilon$ is open. 
\end{proof}
Combining Lemma \ref{l:U3minusQ} and Lemma \ref{l:U3equalQ}, we see that if $b_1$ is not constant, then $(g,\epsilon^{-1}b_1)$ is not Zoll for any $\epsilon\in(0,\epsilon_0]$. This proves the first part of Theorem \ref{t:Zoll1}.

Let us now assume that $b_1$ is a non-zero constant, which we can take to be positive, and suppose that the Gaussian curvature $K:M\to\R$ of $g$ is not constant. We want to show that also in this case there exists $\epsilon_0>0$ such that $(g,\epsilon^{-1}b_1)$ is not Zoll for any $\epsilon\in(0,\epsilon_0]$. The argument is very similar to the one given above for a non-constant function $b_1$ and we give here only a sketch.

Let us fix $\delta_1<\delta_2<\delta_3$ in the non-empty interval $(\min K,\max K)$. We define $U_i:=\{K>\delta_i\}$ for $i=1,2$ and take $U_3$ to be a connected component of $\{K>\delta_3\}$ such that $U_3\cap\{ K=\max K\}\neq\varnothing$. We have a section $W$ of $\pi:SU_1\to U_1$ and we apply Theorem \ref{t:nf} to $U=U_1$ and $U'=U_2$. We get the Hamiltonian function
\[
H_\epsilon=-\epsilon^2(2b_1)^{-1}+\epsilon^4H''_\epsilon,\qquad H''_0=-(2b_1)^{-3}K.
\]
Since $b_1$ is constant, there holds
\begin{equation}\label{e:phiphi''}
X_{H_{\epsilon,\theta}}=\epsilon^4X_{H''_{\epsilon,\theta}},\qquad \varphi_{H_\epsilon}^{\theta_0,\theta_1}=\varphi_{H''_\epsilon}^{\theta_0,\epsilon^4\theta_1},\qquad\forall\,\theta_0\in \T,\ \theta_1\in\R.
\end{equation}
Now Lemma \ref{l:Q}, \ref{l:gin}, \ref{l:U3minusQ}, \ref{l:U3equalQ} hold for this new choice of $U_1$, $U_2$, $U_3$. More precisely, in Lemma \ref{l:gin}, we use \eqref{e:phiphi''} instead of \eqref{e:phiphi} to get
\[
\mathbb A_\epsilon=-\epsilon^2\pi b_1^{-1}+\epsilon^4\big[-2\pi (2b_1)^{-3}K+o(1)\big]
\]
which has an interior minimum in $U_3$ because $K$ has an interior maximum in $U_3$. In Lemma \ref{l:U3minusQ}, we take $\delta_4$ to be a regular value of $K$ in the interval $(\delta_3,\max K)$ and use again \eqref{e:phiphi''} instead of \eqref{e:phiphi}.

This finishes the argument for $b_1$ constant and $K$ non-constant and proves Theorem \ref{t:Zoll1}.

\section{Rotationally invariant Zoll systems}
\label{s:zoll2}

In this section we prove that rotationally invariant Zoll magnetic systems with average $1$ on the two-torus are Zoll-rigid at unbounded sets not contained in $R_{\N}$ but are Zoll-flexible at any of the unbounded discrete sets $\tfrac{1}{k}R_*$ for some $k\in\N$. To this end, we require some preliminaries referring to \cite{Asselle:2019a} for the details. As usual, we write $\T=\R/2\pi\Z$ and consider the cylinder $\T\times \R$ with coordinates $x\in\T$ and $y\in\R$. A rotationally invariant magnetic system $(g,b)$ 
on $\T\times \R$ is then given by
\begin{align*}
g=\diff x^2 + a^2(x) \diff y^2,\quad b =b(x),
\end{align*}
where $a:\T\to (0,+\infty)$ and $b:\T\to \R$ are smooth functions. Taking the quotient by the translation $(x,y)\mapsto (x,y+2\pi)$, we get a magnetic system on $\T^2$ which is Zoll at a speed $r$ if and only if the corresponding system on the cylinder is Zoll at speed $r$. Thus it is enough to consider the magnetic system on the cylinder.

Up to rescaling the $y$-variable $(x,y)\mapsto (x,cy)$ and changing $a$ to $a/c$, we can assume that $a$ integrate to $2\pi$ over $\T$. Requiring the average of $b$ to be $1$ then amounts to having 
\[
\int_{\T} a(x)b(x)\, \diff x=2\pi.
\]
If $r$ is a positive constant, the equations of motion of a $(g,r^{-1} b)$-geodesic $(x,y):\R\to \T\times\R$ whose tangent vector makes an angle $\theta:\R\to \T$ with $\partial_x$ read 
\begin{equation}
\left \{\begin{aligned}\dot x &= \cos \theta, \\
\dot y&=\frac{\sin \theta}{a(x)}, \\
\dot \theta &= r^{-1} b(x) - \frac{a'(x)}{a(x)} \sin \theta , \end{aligned}\right.
\label{equationsofmotion}
\end{equation}
Being invariant under the flow of the vector field $\partial_y$, the system admits by Noether's theorem an additional first integral. Setting $\T_{2\pi r^{-1}}:=\R/2\pi r^{-1}\Z$ and letting $B:\T\to\T$ be any function with $B'(x)=a(x)b(x)$, the first integral is given by 
\begin{equation}
I:\T\times\T \to \T_{2\pi r^{-1}},\quad I(x,\theta ) = a(x)\sin \theta - r^{-1} B(x).
\label{firstintegralI}
\end{equation}
Let us now take $g=g_\fl$, namely $a\equiv1$. Then \eqref{equationsofmotion} and \eqref{firstintegralI} reduce to 
\begin{equation*}
\left \{\begin{aligned} \dot x &= \cos \theta, \\
\dot y&=\sin \theta, \\
\dot \theta &= r^{-1} b(x) , \end{aligned}\right.\qquad I(x,\theta) = \sin \theta - r^{-1}B(x).
\end{equation*}
Let us suppose that $b>0$. In this case, the functions $t\mapsto \theta (t)$ and $x\mapsto B(x)$ are strictly monotonically increasing and we can express $x$ and $y$ as functions of $\theta$. In particular, if $I(x,y)=I_0$, then
\[
x=B^{-1}\big(r(\sin\theta-I_0)\big),\qquad \frac{\diff y}{\diff \theta}=r\sin\theta \cdot v\big(r(\sin\theta-I_0)\big),
\]
where
\[
v:\T \to \R,\qquad v(u) :=\frac{1}{b(B^{-1}(u))}.
\]
From this formulae we see that $x$ is $2\pi$-periodic in $\theta$, while there is a shift in the $y$ coordinate when $\theta$ goes from $-\pi/2$ to $\pi/2$ given by the \textit{shift function} $\Delta_r: \T_{2\pi r^{-1}}\to \R$ defined as
\begin{equation}
\Delta_r^b(I_0):=\int_{-\pi/2}^{\pi/2}\frac{\diff y}{\diff \theta} \, \diff \theta =r\int_{-\pi/2}^{\pi/2} \sin\theta \cdot v\big(r(\sin\theta-I_0)\big)\, \diff \theta=r^2\int_{-\pi/2}^{\pi/2} \cos^2\theta \cdot v'\big(r(\sin\theta-I_0)\big)\, \diff \theta,
\label{shiftfunction}
\end{equation}
where we used integration by parts. Clearly, $(g,r^{-1}b)$ is Zoll if and only if the shift function $\Delta_r$ vanishes identically. In order to check this condition, it is convenient to rewrite $\Delta_r$ slightly by making the change of variables $u=\sin\theta$:
\[
\Delta_r^b(I_0)=r^2\int_{-1}^1  \sqrt{1-u^2}\cdot v'\big (r(u-I_0)\big)\, \diff u.
\]
The function $\Delta_r^b$ has zero average and therefore it vanishes if and only if its Fourier coefficients $\widehat\Delta_r^b(k)$ for $k\in\Z\setminus\{0\}$ vanish. We claim that
\begin{equation}\label{e:fourier}
\widehat{\Delta}_r^b(k)=-i\pi r J_1(kr) \cdot \widehat{v}(-k),
\end{equation}
where $J_1:(0,\infty)\to\R$ is the first Bessel function
\[
J_1(x) = \frac{x}{\pi} \int_{-1}^1 \sqrt{1-u^2} \cdot e^{-ixu}\, \diff u.
\]
Indeed, we have
\begin{align*}
\widehat{\Delta}_r(k) &=r^2\int_{\T_{2\pi r^{-1}}} \int_{-1}^1 \sqrt{1-u^2}\cdot v'\big (r(u-I_0)\big) e^{-i kr I_0}\, \diff u \, \diff I_0\\
&= r^2\int_{-1}^1 \sqrt{1-u^2}\cdot\Big(\int_{\T_{2\pi r^{-1}}} v'\big (r(u-I_0)\big) e^{-i kr I_0}\, \diff I_0\Big)\diff u\\
&= r^2\int_{-1}^1 \sqrt{1-u^2}\cdot\Big(\int_{\T_{2\pi }} v'(s) e^{-i kr(u-r^{-1}s)}\, \diff s\Big)\diff u\\
&=r^2\int_{-1}^1 \sqrt{1-u^2}\cdot e^{-i kru}\cdot\Big(\int_{\T_{2\pi }} v'(s) e^{i ks}\, \diff s\Big)\diff u\\
&=-r^2ik\int_{-1}^1 \sqrt{1-u^2}\cdot e^{-ikr u}\cdot\Big(\int_{\T_{2\pi }} v(s) e^{i ks}\, \diff s\Big)\diff u\\
&= -i\pi r J_1(kr) \cdot \widehat{v}(-k).
\end{align*}
Let us recall the definition of the sets
\[
R_*=\{\xi\in(0,\infty)\ |\ J_1(\xi)=0\},\qquad R_{\N}=\bigcup_{k\in\N}\tfrac{1}{k}R_*.
\]
From the properties of $J_1$, we know that $R_*$ is discrete, bounded from below and unbounded from above, so that $R_{\N}$ is countable and dense in $(0,\infty)$. 

If $r\notin R_{\N}$, then the only real function $v$ satisfying \eqref{e:fourier} for all $k\in\Z$ is the constant function. Therefore,
\begin{equation}\label{e:s*}
\forall\,r\in (0,\infty)\setminus R_{\N},\qquad (g_\fl,r^{-1}b) \text{ is Zoll}\quad \Longleftrightarrow\quad b \text{ is constant}.
\end{equation} 

On the other hand, for every $k\in\N$, let us implicitly define $b_k:\T\to(0,\infty)$ with average $1$ by specifying
\[
v_k(u):=1 + \epsilon\sin (ku),\qquad \text{for some }\epsilon\in(-1,1)\setminus\{0\}.
\]
Then $\widehat v_k(k')=0$ for $|k'|\neq 0,k$ and using \eqref{e:fourier} we conclude that $\widehat\Delta^{b_k}_r=0$ if and only if $r\in \tfrac{1}{k}R_*$. In other words, we have the equivalence
\begin{equation}\label{e:b*}
(g_\fl,r^{-1}b_k)\text{ is Zoll }\quad\Longleftrightarrow\quad r\in \tfrac{1}{k}R_*.
\end{equation}
We are now in position to prove Theorem \ref{t:Zoll2}.

\begin{proof}[Proof of Theorem \ref{t:Zoll2}]
Let $R$ be any subset of $(0,\infty)$ which is unbounded from above and not contained in $R_{\N}$. Let us suppose that $(g,b)$ is a rotationally symmetric magnetic system on $\T^2$ with average $1$ such that $(g,r^{-1} b)$ is Zoll for every $r\in R$. Then \cite[Proposition 2.1]{Asselle:2019b}
implies that $g=g_\fl$ is a flat metric. From \eqref{e:s*} applied to $r\in R\setminus R_{\N}$, we deduce that $b$ is constant. This shows Statement (a). Statement (b) follows by considering the function $b_k$ constructed above and using \eqref{e:b*}.
\end{proof}

\bibliography{_biblio}
\bibliographystyle{plain}

\end{document}